\documentclass[a4paper,10pt]{article}
\pdfoutput=1
\usepackage{amsmath,amssymb,amsfonts,theorem}
\usepackage{graphicx,color}
\usepackage{subfigure}
\usepackage{rotating}
\graphicspath{{./Figures/}}

{ \theorembodyfont{\normalfont} %\theorembodyfont{\rmfamily}
\newtheorem{example}{Example}
\newtheorem{remark}{Remark}
}
\newtheorem{assumption}{Assumption}
\newtheorem{definition}{Definition}
\newtheorem{theorem}{Theorem}
\newtheorem{lemma}{Lemma}
\newtheorem{corollary}{Corollary}

\newcommand{\N}{\mathbb{N}}

\newcommand{\R}{\mathbb{R}}

\newcommand{\K}{\mathcal{K}}
\newcommand{\gammae}{\eta_{ij}}%
\newcommand{\gammaeo}{\eta}%
\newcommand{\dst}{\displaystyle}

\def\be{\begin{equation}}
\def\ee{\end{equation}}
\def\ba{\begin{array}}
\def\ea{\end{array}}
\def\eqa{\begin{eqnarray}}
\def\eqe{\end{eqnarray}}

\author{Claudio De Persis{\thanks{Department of Computer and System Sciences,
Sapienza University of Rome, Italy; {\tt depersis@dis.uniroma1.it} and Faculty of Mathematics and Natural Sciences, University of Groning, 9747 AG Groningen, Netherlands; {\tt c.de.persis@rug.nl}}}, Rudolf Sailer, Fabian Wirth{\thanks{R. Sailer and F. Wirth are with Faculty of Mathematics,
         University of W\"urzburg, 97074 W\"urzburg, Germany
         {\tt \{sailer,wirth\}@mathematik.uni-wuerzburg.de}}}}%}
\title{On a small-gain approach to distributed event-triggered control}
\begin{document}
\maketitle
\begin{abstract}
\noindent In this paper the problem of stabilizing large-scale systems by
distributed controllers, where the controllers exchange
information via a shared limited communication medium is
addressed.  Event-triggered sampling schemes are proposed, where
each system decides when to transmit new information across the
network based on the crossing of some error thresholds. Stability
of the interconnected large-scale system is inferred by applying a
generalized small-gain theorem. Two variations of the
event-triggered controllers which prevent the occurrence of the
Zeno phenomenon are also discussed.
\end{abstract}
\section{Introduction}
We consider large-scale systems stabilized by distributed controllers, which communicate over a limited shared medium.
In this context it is of interest to reduce the communication load.
An approach in this direction is event-triggered sampling, which
attempts to send data only at ``relevant times". In
order to treat the large-scale case, input-to-state
stability (ISS) small-gain results in the presence of
event-triggering decentralized controllers are presented.\\
% An approach in this direction is event-triggered sampling, which attempts only to send "relevant" data.
%In order to treat the large-scale case, a combination of ISS small-gain results with ideas from event-triggering is presented.
The stability (or stabilization) of large-scale interconnected
systems is an important problem which has attracted much  interest.
In this context the small-gain theorem was extended to the interconnection of several
${\cal L}_p$-stable subsystems. Early accounts of this approach
are \cite{vidyasagar1981} (see also \cite{Siljak1978large}) and
references therein. For instance, in \cite{vidyasagar1981},
Theorem 6.12, the influence of each subsystem on the others is
measured via an ${\cal L}_p$-gain, $p\in [1,\infty]$ and the
${\cal L}_p$-stability of the interconnected  system holds
provided that the spectral radius of the matrix of the gains is
strictly less than unity. In other words, the stability of
interconnected ${\cal L}_p$-stable systems holds under a condition
of weak coupling.\\
In the nonlinear case a notion of robustness
with respect to exogenous inputs is input-to-state stability (ISS) (\cite{sontag_smooth_1989}).
%(\cite{sontag.wang.scl95}).
If in a large-scale system each subsystem is ISS, then the influence between the subsystems
is typically modeled via nonlinear gain functions. Small-gain
theorems have been developed for ISS systems as well
(\cite{jiang_small-gain_1994,jiang.et.al.aut96,teel.tac96}) and more recently
they have been extended to the interconnection of several ISS
subsystems (\cite{dashkovskiy2007iss,drw}).
For a recent comprehensive discussion about the literature on ISS small-gain results see \cite{liu.et.al.nolcos2010}.\\
In the literature on large-scale systems we have discussed so far,
the communication aspect does not play a role. If however, a
shared communication medium leads to significant further
restrictions, concepts like event-triggering become of interest.
We speak of event-triggering if the occurrence of predefined
events, as e.g. the violation of error bounds, triggers a
communication attempt. Using this approach a decentralized way of
stabilizing large-scale networked control systems
which are finite ${\cal L}_p$-stable has been
proposed in \cite{wang2009event,wang2011}. In these papers each
subsystem broadcasts information when a locally-computed error
signal exceeds a state-dependent threshold. Similar ideas are
presented in \cite{tabuada.tac07,wang2009self}. Numerical
experiments e.g., \cite{wang2009self} show that event-triggered
stabilizing controllers can lead to less information transmission
than
standard sampled-data controllers. For consensus problems, event-triggered
controllers are studied in \cite{dim.johan.cdc09}.
\\
One drawback of the proposed event-triggered sampling scheme is the need for constantly checking the validity of an inequality.
A related approach which tries to overcome this issue is termed self-triggered sampling (see e.g., \cite{anta2009sample,mazo2010iss}).\\
{}From a more general perspective, the way
in which the subsystems access the medium must be carefully
designed. In this paper we do not discuss the problem of collision avoidance. This problem is addressed for instance in the literature on medium access protocols, such as the
round-robin and the try-once-discard protocol. E.g., in \cite{nesic.teel.tac04} a large
class of medium access protocols are treated as
dynamical systems and the stability analysis  in the presence of
communication constraints is carried out by including the
protocols in the closed-loop system. This allows to give an estimate on the maximum allowable transfer
interval (MATI), that is the maximum interval of time between two
consecutive transmissions which the system can tolerate without
going into instability. The advantage of event-triggering lies in the possibility of reducing overall communication load. However, if events occur simultaneously at several subsystems the problem of collision avoidance remains. We will discuss this in future work.\\
The purpose of this paper is to explore event-triggered
distributed controllers for systems which are given as an
interconnection of a large number of ISS subsystems.
Since input-to-state stability  and finite ${\cal L}_p$ stability
are distinct properties for nonlinear systems, the class of
systems under consideration in this paper differs from the one in
\cite{wang2009event,wang2011}. Moreover, we use analytical tools
which have been extended to deal with other classes of systems
(such as integral-input-to-state stable systems
\cite{ito.et.al.cdc09} and hybrid systems
\cite{liu.et.al.nolcos2010}), and therefore the arguments in this
paper are potentially applicable to a larger class of systems than
the one actually considered here.\\
We assume that the gains
measuring the degree of interconnection satisfy a generalized
small-gain condition. To simplify presentation, it is assumed
furthermore that the graph modeling the interconnection structure
is strongly connected. This assumption can be removed as in
\cite{drw}. Since our event-triggered implementation of the
control laws introduces disturbances into the system, the ISS
small-gain results available in the literature are not applicable.
An additional condition is required for general nonlinear systems
using event-triggering. This condition is explicitly given in the
presented general small-gain theorem. Moreover, the functions
which are needed to design the state-dependent triggering
conditions are explicitly designed in such a way that the
triggering events which supervise the broadcast by a subsystem
only depend on local information. As an introductory example we
explicitly discuss the special case of linear systems, although
for this class of systems the techniques of \cite{wang2009event,wang2011} are applicable.\\
As distributed event-triggered controllers can potentially require
transmission times which accumulate in finite time,
we also discuss two variations of the proposed small-gain
event-triggered control laws which prevents the occurrence of the
Zeno phenomenon. Related papers
are also \cite{donkers.heemels.ifac11}, \cite{mazo.tabuada.arxiv2010}.\\
Section \ref{s.prelim} presents the class of system we focus our
attention on, along with a number of preliminary notions and standing
assumptions. The definition of the term event-triggered control can be found in Section~\ref{sec:trig}.\\% and our first result. Based on this theorem,
In Section~\ref{sec:lyap} the notion of ISS-Lyapunov functions is presented.
Based on this notion
small-gain event-triggered distributed controllers are discussed
in Section \ref{s.detc}. The results are particularized to
the case of linear systems in Section~\ref{sec:ex} along with a few simulation results in
Section \ref{sect.linear}. A nonlinear example together with simulation results is discussed in Section~\ref{sec:nonlinear}.\\
The Zeno-free distributed
event-triggered controllers are proposed in Section
\ref{sec.towards}. The last section contains the conclusions of
the paper.\\

\noindent {\bf Notation} $\mathbb{N}_0=\mathbb{N}\cup \{0\}$.
$\R_+$ denotes the set of nonnegative
real numbers, and $\R_+^n$ the nonnegative orthant, i.e.\ the set
of all vectors of $\R^n$ which have all entries nonnegative.
By $||\cdot||$ we denote the Euclidean norm of a vector or a matrix.
\\
A function $\alpha:\R_+\to\R_+$ is a  class-${\cal K}$ function if
it is continuous, strictly increasing and zero at zero. If it is
additionally unbounded, i.e.\ $\lim_{r\to+\infty} \alpha(r)
=\infty$, then $\alpha$ is said to be a  class-${\cal K}_\infty$
function. We use the notation $\alpha \in  {\cal K}$ ($\alpha \in
{\cal K}_\infty$) to say that $\alpha$ is a class-${\cal K}$
(class-${\cal K}_\infty$) function. The symbol ${\cal K}\cup
\{0\}$ (${\cal K}_\infty\cup \{0\}$) refers to the set of
functions which include all the class-${\cal K}$ (class-${\cal
K}_\infty$) functions and the function which is identically zero.
A function $\alpha:\R_+\to\R_+$ is positive definite if
$\alpha(r)=0$ if and only if $r=0$.
We denote the right-hand limit by $\lim_{t\searrow \tau}x(t)=x(\tau^+)$.
\section{Preliminaries}\label{s.prelim}
Consider the interconnection of $N$ systems described by equations
of the form:
\be\label{interconnected.system}
\ba{rcl}
\dot x_i &=& f_i(x, u_i)\\
u_i &=& g_i(x+e)\;,
\ea\ee
where $i\in {\cal N}:=\{1,2,\ldots, N\}$,
$x=(x_1^\top\ldots x_N^\top)^\top$, with $x_i\in \R^{n_i}$, is the state
vector and $u_i\in \R^{m_i}$ is the $i$th control input. The
vector $e$, with $e=(e_1^\top\ldots e_N^\top)^\top$ and $e_i\in \R^{n_i}$, is
an error affecting the state. We shall assume that the maps $f_i$
satisfy appropriate conditions which guarantee existence and
uniqueness of solutions for ${\cal L}_\infty$ inputs $e$.
In particular, the $f_i$ are continuous. Also we assume that the
$g_i$ are locally bounded, i.e. for each compact set
$K\subset\R^n$ ($n:=\sum_{i=1}^N n_i$) there exists a constant
$C_K$ with $\|g_i(x)\|\leq C_K$ for each $x\in K$.\\
The interconnection of each system $i$ with another system $j$ is
possible in two ways. One way is that the system $j$ influences
the dynamics of the system $i$ directly, meaning that  the state
variable $x_j$ appears non trivially in the function $f_i$. The
other way is that the controller $i$ uses information from
system $j$. In this case, the state variable $x_j$ appears non
trivially in the function $g_i$ (and affects indirectly the
dynamics of the system $i$).\\
In this paper we adopt the notion of ISS-Lyapunov functions
(\cite{sontag.wang.scl95}) to model the interconnection among the
systems.
\begin{definition}
\label{def:iss}
A smooth function $V\;:\;\R^n\rightarrow\R_+$ is called an \emph{ISS-Lyapunov} function for system $\dot x=f(x,u)$ if there exist $\alpha_1,\,\alpha_2\in\K_\infty$ and $\alpha_3,\,\chi\in\K$, such that for any $x\in\R^n$
\[
\alpha_1(\|x\|)\leq V(x)\leq\alpha_2(\|x\|)
\]
and the following implication holds for all $x\in\R^n$ and all admissible $u$
\[
   V(x)\geq\chi(\|u\|)\Rightarrow \nabla V(x)f(x,u)\leq -\alpha_3(\|x\|)\;.
\]
\end{definition}
It is well known that a system as in Definition~\ref{def:iss} is ISS if and only if it admits an ISS-Lyapunov function.
If there are more than one input present in the system, the question how to compare the influence of the different inputs arises.
To answer this question we preliminary recall the notion of
monotone aggregate functions from \cite{drw}:\footnote{In the
definition below, for any pair of vectors $v,z\in \R^n$, the
notations $v\ge z$, $v>z$ are used to express the property that
$v_i\ge z_i$, $v_i>z_i$ for all $i=1,2,\ldots, n$. Moreover, the
notation $v\gneqq z$ indicates that $v\ge w$ and $v\ne w$.}
\begin{definition}\label{def.maf}
A continuous function $\mu: \R^n_+\to \R_+$ is a {\em monotone
aggregation function} if:
\begin{description}
\item{(i)} $\mu(v)\ge 0$ for all $v\in \R^n_+$ and $\mu(v)>0$ if $v\gneqq
0$;
\item{(ii)} $\mu(v)> \mu(z)$ if $v>z$;
\item{(iii)} If $||v||\to \infty$ then $\mu(v)\to \infty$.
\end{description}
The space of monotone aggregate functions (MAFs in short) with
domain $\R^n_+$ is denoted by $MAF_n$. Moreover, it is said that
$\mu\in MAF_n^m$ if for each $i=1,2,\ldots, m$, $\mu_i\in MAF_n$.
\end{definition}
Monotone aggregate function are used in the following assumption
to specify the way in which systems are interconnected and how
controllers use information about the other systems:
\begin{assumption}\label{a1}
For $i=1,2,\ldots, N$,  there exists a differentiable function
$V_i: \R^{n_i}\to \R_{+}$, and class-${\cal K}_\infty$ functions
$\alpha_{i1}, \alpha_{i2}$ such that
\[
\alpha_{i1}(||x_i||)\le V_i(x_i)\le \alpha_{i2}(||x_i||)\;.
\]
Moreover there exist functions $\mu_i\in {\rm MAF}_{2N}$,
$\gamma_{ij}, \gammae\in {\cal K}_\infty\cup \{0\}$,
%$j\ne i$, $\gamma_i\in {\cal K}\cup \{0\}$,
$\alpha_i$ positive
definite such that
\begin{equation}
    \label{aa1}
\begin{aligned}
    V_i(x_i)\ge &\, \mu_i(\gamma_{i1}(V_1(x_1)),
\ldots,\gamma_{iN}(V_N(x_N)),
\gammaeo_{i1}(||e_1||), \ldots,
\gammaeo_{iN}(||e_N||))
\\
\Rightarrow &\, \nabla V_i(x_i) f_i(x, g_i(x+e))\le -
\alpha_i(||x_i||)\;.
\end{aligned}
\end{equation}
\end{assumption}\vspace{-5pt}
Loosely speaking, the function $\gamma_{ij}$ describes the overall
influence of system $j$ on the dynamics of system $i$, while the
function $\gammae$ describes the influence of the system $j$ on
the system $i$ via the controller $g_i$. In particular,
$\gammae\ne 0$ if and only if the controller $u_i$ is using
information from the system $j$. In this regard $\gammae$
describes the influence of the imperfect knowledge of the state of
system $j$ on system $i$ caused by e.g., measurement noise. On the
other hand, if $i\ne j$ and $\gamma_{ij}\ne 0$, then the system
$j$ influences the system $i$ (either explicitly or implicitly).
We assume that $\gamma_{ii}=0$ for any $i$. Observe that if the
system $i$ is not influenced by any other system $j\ne i$, and
there is no error $e_i$ on the state information $x_i$ used in the
control $u_i$, then the assumption amounts to saying that the
system $i$ is input-to-state stabilizable via state feedback.
\begin{remark} In general it is hard to design
controllers that render the closed loop system ISS as we demand in
Assumption~\ref{a1}. Though, there exist design techniques for
special classes of systems. See e.g.,
\cite{krstic1995nonlinear,liberzon2002universal,fah.cocv99} and
the references therein.
\end{remark}
For future use we denote the set of states entering
the dynamics of system $i$ by
\[
\Sigma(i)=\{j\in\mathcal{N}\;:\;f_i \text{ depends explicitly on
}x_j\}\,,
\]
%FRW
where explicit dependence of $f_i$ on $x_j$ means that $�\partial f_i
/ \partial x_j \not \equiv 0$. Similarly for the controllers we denote
\[
C(i)=\{j\in\mathcal{N}\;:\; g_i\text{ depends explicitly on
}x_j\}\;.
\]
It is also convenient to define the set of the controllers to
which the state of system $i$ is broadcast
\[
Z(i)=\{j\in\mathcal{N}\;:\;g_j\text{ depends explicitly on
}x_i\}\;.
\]
\subsection{The case of linear systems}
\label{sec:ex}
To get acquainted with the assumption above, we examine in the
following example the case in which the systems are linear.
\begin{example}
%\label{sec:ex}
Consider the interconnection of $N$ linear subsystems
\[\ba{rcl}
\dot x_i &=& \sum_{j=1}^N A_{ij} x_j + B_i u_i\\
u_i &=& \sum_{j=1}^N K_{ij} (x_j+e_j)\;.
\ea\]
For each index $i$, we assume that the pairs $(A_{ii}, B_i)$ are
stabilizable and we let the matrix $K_{ii}$ be such that $\bar
A_{ii}:=A_{ii}+B_i K_{ii}$ is Hurwitz. Then for each $Q_i=Q_i^\top>0$
there exists a matrix $P_i=P_i^\top>0$ such that $\bar A_{ii}^\top
P_i+P_i\bar A_{ii}=-Q_i$ leading to Lyapunov functions $V_i(x_i)=x_i^\top P_ix_i$.\\ We consider now the expression $\nabla
V_i(x_i) \dot x_i$ where
\[\ba{rcl}
\dot x_i &=& \sum_{j=1}^N (A_{ij}+ B_i K_{ij}) x_j + \sum_{j=1}^N
B_i K_{ij} e_j\\
&=:& \sum_{j=1}^N \bar A_{ij} x_j + \sum_{j=1}^N \bar B_{ij} e_j\;,
\ea\]
with $\bar B_{ij}:=B_iK_{ij}$ and $\bar A_{ij}:=A_{ij}+B_iK_{ij}$.\\
Standard calculations lead to
\begin{multline*}
\nabla V_i(x_i) \dot x_i \le -c_i ||x_i||^2 + 2 ||x_i||\,||P_i||
\Big(\sum_{j=1, j\ne i}^N ||\bar A_{ij}||\, ||x_j||+ \sum_{j=1}^N
||\bar B_{ij}||\, ||e_j||\;\Big)\;,
\end{multline*}
where\footnote{For symmetric $Q_i$ we let $\lambda_{\min}(Q_i)$ denote the
smallest eigenvalue of $Q_i$.} $c_i=\lambda_{\min}(Q_i)$.
Moreover, for any $0< \tilde c_i< c_i$ the inequality
\[
||x_i||\ge \frac{2||P_i||}{\tilde c_i} \Big(\sum_{j=1, j\ne i}^N
||\bar A_{ij}||\, ||x_j||+ \sum_{j=1}^N ||\bar B_{ij}||\, ||e_j||\;\Big)
\]
implies that
\[
\nabla V_i(x_i) \dot x_i \le -(c_i-\tilde c_i) ||x_i||^2\;.
\]
The former inequality is implied by
%FRW: small adjustment to make line fit
\begin{multline*}
V_i(x_i)\ge ||P_i||^3\,\cdot \Big[\frac{2}{\tilde c_i} \Big(\sum_{j=1, j\ne i}^N
\frac{||\bar A_{ij}||}{[\lambda_{\min}(P_j)]^{1/2}}\, V_j(x_j)^{1/2}
+ \sum_{j=1}^N ||\bar B_{ij}||\, ||e_j||\;\Big)\Big]^2\;.
\end{multline*}
We conclude that (\ref{aa1}) holds with
\begin{equation}
 \label{eq:gain_lin}
 \left.
\ba{rcl}
\gamma_{ii}&=& 0\\
\gamma_{ij}(r) &=& \frac{2||P_i||^{3/2}}{\tilde c_i}\frac{||\bar A_{ij}||}{[\lambda_{\min}(P_j)]^{1/2}}\, r^{1/2}\\
\gammae(r) &=& \frac{2||P_i||^{3/2}}{\tilde c_i}||\bar B_{ij}||\, r\\
\mu_i(s) &=&\Big(\sum_{j=1}^{2n} s_j\;\Big)^2\\
\alpha_i(r)&=&({c_i-\tilde c_i}) r^2\;.\\
\ea
\right\}
\end{equation}
It is important to remember that not all the functions
 $\gamma_{ij}$ and $\gammae$ are non-zero. Namely,
 $\gamma_{ij}$ ($i\neq j$) is non-zero if and only if $\bar A_{ij}$
is a non-zero matrix.
Similarly, $\eta_{ij}\neq 0$ if and only if $\bar B_{ij}\neq 0$.
\end{example}

\medskip
\section{Event-triggered control}
\label{sec:trig}
In this paper we investigate event-triggered control schemes. Such schemes (or similar) have been studied in \cite{anta2009sample,mazo2010iss,tabuada.tac07,wang2009event,wang2009self,wang2011}.\\
We consider systems as defined in \eqref{interconnected.system}.
Combined with a triggering scheme the setup under consideration has the form
\begin{align}
\label{def:trigsys}
 \notag\dot x_i&=f_i(x,u_i)\\
\notag u_i&=g_i(x+e)\\
\dot{\hat x}&=0\\
\notag e&=\hat x-x\\
\intertext{with triggering condition}
 T_i&(x_i,e_i)\geq 0\;.
\end{align}
Here $x_i$ is the state of system $i\in\mathcal{N}$, $\hat x$ is the information available at the controller and the controller error is $e=\hat x-x$.
We assume that the triggering function $T_i$ are
jointly continuous in $x_i,\,e_i$
and satisfy $T_i(x_i, 0)<0$ for all $x_i\neq0$.\\
Solutions to such a triggered feedback are defined as follows. We
assume that the initial controller error is
$e_0=0$. Given an initial condition $x_0$ we define
\[t_1:=\inf\{t>0\;:\;\exists i\in\mathcal{N}\;{ \rm s.t. }\;T_i(x_i(t),e_i(t))\geq 0\}\;.\]
At time instant $t_1$ the systems $j$ for which $T_j(x_j,e_j)=0$ broadcast their respective state $x_j$ to all controllers with $j\in C(i)$. In particular, $e_j(t_1^+)=0$ for these indices $j$.\\
Then inductively we set for $k=1,2,\dots$
\[t_{k+1}:=\inf\{t>t_k\;:\;\exists i\in\mathcal{N}\;{ \rm s.t. }\;T_i(x_i(t),e_i(t))\geq0\}\;.\]
We say that the triggering scheme induces Zeno behavior if for a given initial condition
$x_0$ the event times $t_k$ converge to a finite $t^*$.

%\newpage

\begin{remark}
~\\[-4ex]
\label{rem:trigg}
 \begin{itemize}
\item One of the proposed
triggering schemes in this paper uses the information
$d_i$ which is an estimate of $\|\dot x_i\|$ available at system $i$.
 For this scheme the
triggering condition will be replaced by $T_i(x_i,e_i,d_i)\ge 0$.
\item The condition $e_0=0$ is used for simplicity. The triggering scheme uses implicitly that system $i$ knows its state $x_i$ and the error at the controller $e_i$ (and possibly the estimate $d_i$ if this is used). It would therefore be sufficient to have an initial condition where system $i$ is aware of $e_{0i}=\hat x_i(0)-x_i(0)$. However, such an assumption is most likely guaranteed by an initial broadcast of all states of the subsystems. But then $e_0=0$ is plausible.
\item
%\margin{Here some references on how to cope with effects like delay, collisions, packet loss... are missing. Are you aware of some good references? \\
%C: I added one reference}
 It is a standing assumption in
this paper that information transmission is reliable, so that
broadcast information is received instantaneously and error free
by the controllers. If this is not the case, additional techniques
as studied e.g.\ in \cite{tiwari.et.al} have to be employed. This
will be the topic of future research.
\item In many useful triggering conditions we have that $T_i(0,0,d_i)=0$.
If the system were to remain at $x=0$ this would lead to a
continuum of triggering events, which do not provide information.
To avoid this (academic) problem we propose to add the condition
%\margin{Rudi, is it easy to detect in the simulations when one of
%the states has reached the origin and then stop the update?}
%\margin{Rudi: I don't know. In numerics you hardly reach zero. Hence this case never happend to me.}
 that information is broadcast
once $x_i$ reaches the state zero, but no further transmission by
system $i$ occurs as long as it stays at zero.
\item For simplicity, we assume $\dot{\hat x}=0$ in between triggering times. Usually, this is referred to as zero order hold.\\
Other techniques are also possible, which could lower the
triggering frequency.
Consider for instance the case that each controller has a model for the dynamics of each other subsystem. Then each controller could use these models to calculate $\hat x$ rather than keeping it constant. Another approach would be to extrapolate $\hat x$ linearly with the help of the last values for $\hat x$. This is known as predictive first order hold.
%\margin{Should we add a reference?\\Rudi: in my opinion, a reference is not really needed. I also have to admit that I am not aware of one within the context of large scale systems.}
Both techniques would lead to $\dot{\hat x}\neq0$. The considerations in this paper would also hold true for these cases with slight modifications of the proofs.
%\margin{Rudi: The sentence was not good. I changed it to a hopefully better one.}

 \end{itemize}

\end{remark}

\section{ISS Lyapunov functions for large-scale systems}
\label{sec:lyap}
In this section we review a general procedure for the construction of ISS Lyapunov functions. In particular, we extend recent results to a more general case that covers the case of event-triggered control.\\
Condition (\ref{aa1}) can be used to naturally build a graph which
describes how the systems are interconnected. Let us introduce the
matrix of functions $\Gamma \in ({\cal K}_\infty\cup
\{0\})^{N\times N}$ defined as
\[
\Gamma=\left(\ba{ccccc}
0 & \gamma_{12} & \gamma_{13} & \ldots & \gamma_{1N}\\
\gamma_{21} & 0 & \gamma_{23} & \ldots & \gamma_{2N}\\
\vdots & \vdots & \vdots & \ddots & \vdots\\
\gamma_{N1} & \gamma_{N2} & \gamma_{N3} & \ldots & 0
\ea
\right)\;.
\]
Following \cite{drw}, we associate to $\Gamma$ the adjacency
matrix $A_\Gamma=[a_{ij}]\in \{0,1\}^{N\times N}$ whose entry
$a_{ij}$ is zero if and only if $\gamma_{ij}=0$, otherwise it is
equal to $1$. $A_\Gamma$ can be interpreted as the adjacency
matrix of the graph which has a set ${\cal N}$ of $N$ nodes, each
one of which is associated to a system  of
(\ref{interconnected.system}), and a set of edges ${\cal
E}\subseteq {\cal N}\times {\cal N}$ with the property that
$(j,i)\in {\cal E}$ if and only if $a_{ij}=1$. Recall that a graph
is strongly connected if and only if the associated
adjacency matrix is irreducible. In the present case, if the
adjacency matrix $A_\Gamma$ is irreducible, then we say that
$\Gamma$ is irreducible. In other words, the matrix of functions
$\Gamma$ is said to be irreducible if and only
if the graph associated to it is strongly connected.
For later use, given $\mu_i\in {\rm MAF}_{N}$, $\gamma_{ij}\in
{\cal K}_\infty\cup \{0\}$, it is useful to introduce the map
$\Gamma_\mu:\R_+^{N}\to \R_+^N$ defined as
\[
\Gamma_\mu(r)=\left(\ba{c}
\mu_1(\gamma_{11}(r_1), \ldots, \gamma_{1N}(r_N))\\
\vdots\\
\mu_N(\gamma_{N1}(r_1), \ldots, \gamma_{NN}(r_N))\\
\ea\right).
\]
Since the functions which describe the interconnection of the
system are in general nonlinear, the topological property of
graph connectivity may not be sufficient to ensure stability
properties of the interconnected system. There must also be a way
to quantify the degree of coupling of the systems. In this paper,
this is done using the following notion:
\begin{definition}
A map $\sigma\in {\cal K}_\infty^N$ is an
$\Omega$-path with respect to $\Gamma_\mu$ if:
\begin{description}
\item{(i)} for each $i$, the function $\sigma_i^{-1}$ is locally Lipschitz continuous on $(0,\infty)$;
\item{(ii)} for every compact set $K\subset (0,\infty)$ there are constants $0<c<C$ such that for all $i=1,2,\ldots, N$
and all points of differentiability of $\sigma_i^{-1}$ we
have:
\[
0<c\le (\sigma_i^{-1})'(r)\le C\;,\quad \forall r\in K;
\]
\item{(iii)} $\Gamma_\mu(\sigma(r))<\sigma(r)$ for all $r>0$.
\end{description}
\end{definition}
Condition (iii) in the definition above amounts to a small-gain
condition for large-scale non-linear systems (in other words,
condition (iii) requires the degree of coupling among the
different subsystems to be weak. For a more thorough discussion on condition (iii) see \cite{drw}). To familiarize with the condition, take the case $N=2$
and $\mu_1=\mu_2=\max$ (it is not difficult to see that the
function $\max_{1\le i\le N} r_i$ belongs to $MAF_N$). Then
\[
\Gamma_\mu(r)=\left(\ba{c} \gamma_{12}(r_2)\\
\gamma_{21}(r_1) \ea \right)\;.
\]
We want to show that there exists $\sigma\in {\cal K}_\infty^2$
such that $\Gamma_\mu(\sigma(s))<\sigma(s)$ for all $s>0$ if and
only if $\gamma_{12}\circ \gamma_{21}(r)<r$ for all $r>0$
(the latter can be viewed as a small-gain condition for the
interconnection of two ISS-subsystems). To this purpose, choose
\[
\sigma(s)=\left(\ba{c} s\\
\sigma_{2}(s) \ea \right)\;,
\]
where $\gamma_{21}<\sigma_{2}<\gamma_{12}^{-1}$. As a consequence
of this choice, $\Gamma_\mu(\sigma(s))$ becomes:
\[
\Gamma_\mu(\sigma(s))=\left(\ba{c} \gamma_{12}(\sigma_2(s))\\
\gamma_{21}(s) \ea \right)\;.
\]
By construction, $\gamma_{12}(\sigma_2(s))<s=\sigma_1(s)$ and
$\gamma_{21}(s)<\sigma_{2}(s)$, i.e.
$\Gamma_\mu(\sigma(s))<\sigma(s)$ for all $s>0$.\\
Strong connectivity of $\Gamma$ and an additional condition
implies a weak coupling among all the systems, in the following
sense (see \cite{drw} for a proof and a more complete statement):
\begin{theorem}
Let $\Gamma \in ({\cal K}_\infty\cup \{0\})^{N\times N}$ and
$\mu\in \operatorname{MAF}_{N}^N$. If $\Gamma$ is irreducible and $\Gamma_\mu
\not \ge id$ \footnote{$\Gamma_\mu \not \ge id$ means that for all
$s\neq0$ $\Gamma_\mu(s) \not \ge s$, i.e.\ for all $s\in \R_+^N$ such
that $s\neq0$ there exists $i\in {\cal N}$ for which $\mu_i(s_1,
\ldots, s_N)<s_i$. } then there exists an
$\Omega$-path $\sigma$ with respect to $\Gamma_\mu$.
\end{theorem}
\begin{remark}
In fact, the irreducibility condition on $\Gamma$ is a purely technical assumption. A way how to relax it can be found in \cite{drw}.
\end{remark}
The small gain condition stated above is reformulated
in the following assumption
to take into account the case in which the error inputs
are present in the system:
\begin{assumption}\label{a2}
 There exist an $\Omega$-path $\sigma$ with respect to $\Gamma_\mu$ and a map $\varphi\in ({\cal K}_\infty
\cup \{0\})^{N\times N}$ such
that:
\be\label{omega.condition}
\overline{\Gamma}_\mu(\sigma(r), \varphi(r))<\sigma(r)\;,\quad
\forall r>0\;,
\ee
where $\overline{\Gamma}_\mu(\sigma(r),
\varphi(r))$ is defined by
%
%\begin{multline*}
%\left(\ba{c}
%%
%\mu_1(\gamma_{11}(\sigma_1(r)),..,\gamma_{1N}(\sigma_N(r)), \varphi_{11}(r),.., \varphi_{1N}(r))\\
%\vdots  \\
%\mu_N(\gamma_{N1}(\sigma_1(r)),..,\gamma_{NN}(\sigma_N(r)), \varphi_{N1}(r),.., \varphi_{NN}(r))\\
%\ea\right).
%\end{multline*}
\begin{multline*}
\overline{\Gamma}_\mu(\sigma(r),
\varphi(r)):= \left(\ba{c}
\mu_1(\gamma_{11}(\sigma_1(r)),..,\gamma_{1n}(\sigma_N(r)), \varphi_{11}(r),.., \varphi_{1N}(r))\\
\vdots  \\
\mu_N(\gamma_{N1}(\sigma_1(r)),..,\gamma_{NN}(\sigma_N(r)), \varphi_{N1}(r),.., \varphi_{NN}(r))\\
\ea\right).
\end{multline*}
\end{assumption}

\begin{remark}
We remark that in the case $\mu_i=\max$ for each $i\in\mathcal{N}$, one can exploit the degree of freedom given by $\varphi$ in
such a way that the condition (\ref{omega.condition}) boils down
to the small-gain condition $\Gamma_\mu(r)\ngeq r$. In fact, once an
$\Omega$-path has been determined, if the small-gain
condition is true then it suffices to choose $\varphi_{ij}$ such
that, for any $i,j\in \mathcal{N}$, $\varphi_{ij}\le
\gamma_{ik}\circ \sigma_k$ for some $k \in \mathcal{N}$. For
a more general discussion on the fulfillment of
\eqref{omega.condition} as a consequence of the small-gain
condition, we refer the interested reader to \cite{drw},
Corollaries 5.5-5.7.
\end{remark}
\begin{remark}
 \label{rem:eta}
Observe that $\gamma_{ij}$ describes the influence of
%\margin{Rudi: Sorry, my fault.}
system $j$ on the
dynamics of system $i$ either directly or through its controller $g_i$. Hence for $i\neq j$ the gains $\gamma_{ij}\neq 0$ if and only if
$j\in\Sigma(i)$ or $j\in C(i)$.
Analogously, $\eta_{ij}\neq 0$ if and only if $j\in C(i)$,
meaning that the controller $i$ depends explicitely on the state
of system $j$. Because the $\varphi_{ij}$ from Assumption~\ref{a2} describe the gains for the error input, there is no loss in generality if we set conventionally $\varphi_{ij}=0$ if $\eta_{ij}=0$. From the definition of $C$ and $Z$ it is evident
that $j\in C(i)$ is equivalent to $i\in Z(j)$.
\end{remark}

\section{Main results}\label{s.detc}

In our first result it is shown that a Lyapunov function $V$  and
a set of decentralized conditions exist which guarantee that $V$
decreases along the trajectories of the system:

\begin{theorem}\label{t1}
Let Assumptions \ref{a1} and \ref{a2} hold. Let $V(x)=\max_{i\in
{\cal N}} \sigma_i^{-1}(V_i(x_i))$ and, for each $j\in {\cal N}$,
define:
\begin{equation}
\label{eq:threestar}
\chi_j=\sigma_j\circ \hat\gammaeo_j\;,\;
\mbox{with}\;
\hat\gammaeo_j=\max_{i\in {Z(j)}}
\varphi_{ij}^{-1}\circ\gammaeo_{ij}\;.
\end{equation}

Then there exist a positive definite
$\alpha\,:\,\R_+\rightarrow\R_+$ such that the condition
\be\label{str}
V_i(x_i)\ge \chi_i(||e_i||),\quad\forall\, i\in {\cal N}
\ee
implies
\[
\langle p,f(x,g(x+e))\rangle\le -\alpha(||x||),%+\gamma(|w|)
\; \forall p \in \partial V(x)\;,
\]
where $\partial V$ denotes the Clarke generalized
gradient\footnote{We recall that by Rademacher's theorem the
gradient $\nabla V$ of a locally Lipschitz function $V$ exists
almost everywhere. Let $N$ be the set of measure zero where
$\nabla V$ does not exist and let $S$ be any measure zero subset
of the state space where $V$ lives. Then $\partial V(x)={\rm
co}\{\lim_{i\to +\infty}\nabla V(x_i):\, x_i\to x,\; x_i\not\in
N\; x_i\not \in S\}$. } and %%

\[
f(x,g(x+e)) = \left(\ba{c}
f_1(x, g_1(x+e))\\
\ldots\\
f_n(x, g_n(x+e))\\
\ea\right)\;.
\]

\end{theorem}

\medskip

\begin{proof}For each $x$, let ${\cal N}(x)\subseteq {\cal N}$ be the set of indices $i$ for which
$V(x)=\sigma_i^{-1}(V_i(x_i))$. Let $i\in {\cal N}(x)$ and set
$r=V(x)$. Then
\be\label{c1}\ba{rl}
V_i(x_i)=&\sigma_i(r)>
\overline{\Gamma}_{\mu,i}(\sigma(r),
\varphi(r))\\
 =&\mu_i(\gamma_{i1}(\sigma_1(r)),..,\gamma_{iN}(\sigma_N(r)),
\varphi_{i1}(r),.., \varphi_{iN}(r)). \ea\ee
Observe that by definition of $V(x)$, for any $i\in {\cal N}(x)$ and
any $j\in {\cal N}$,
\begin{equation}
\label{gammaij}
\gamma_{ij}(\sigma_j(r))=\gamma_{ij}(\sigma_j(V(x))) \ge
\gamma_{ij}(\sigma_j(\sigma_j^{-1}(V_j(x_j)))=\gamma_{ij}(V_j(x_j))\;.
\end{equation}
Note that  for $j\notin C(i)$ we have $\varphi_{ij}=0$ and
$\eta_{ij}=0$. Hence for $j\notin C(i)$ it holds trivially that
\begin{equation}
 \label{eq:triv}
\varphi_{ij}(r)\geq \eta_{ij}(\|e_j\|).
\end{equation}
 This is also true if $j\in C(i)$ (or equivalently $i\in Z(j)$).
In fact, since for any $j\in {\cal N}$,
\[
V_j(x_j)\ge \chi_j(||e_j||)\;,\quad \chi_j=\sigma_j\circ
\hat \gammaeo_{j}
\]
we have, using the definition of $V$,  \eqref{str} and
\eqref{eq:threestar}, that
\begin{multline}\label{varphi}
\varphi_{ij}(r)=\varphi_{ij}(V(x))\ge
\varphi_{ij}(\sigma_j^{-1}(V_j(x_j))) 
\ge \varphi_{ij}(\sigma_j^{-1}\circ
\sigma_j(\hat \gammaeo_{j}(||e_j||))) \\
\ge \varphi_{ij}(\sigma_j^{-1}\circ
\sigma_j(
\varphi_{ij}^{-1}\circ\gammae(||e_j||)))
= \gammae(||e_j||)\;.
\end{multline}
Observe that $\mu_i(v)\ge \mu_i(z)$ for all $v\ge z\in \R_+^{2N}$ since $\mu_i\in MAF_{2N}$
and as a consequence of Definition \ref{def.maf}, (ii). Since
$r=V(x)\ge \sigma_i^{-1}(V_i(x_i))$ for all $i\in {\cal N}$,
by (\ref{gammaij}), \eqref{eq:triv} and (\ref{varphi}),
\begin{multline}
\mu_i(\gamma_{i1}(\sigma_1(r)),\ldots,\gamma_{iN}(\sigma_N(r)),
\varphi_{i1}(r),\ldots, \varphi_{iN}(r)) \ge\\
\mu_i(\gamma_{i1}(V_1(x_1)),\ldots,\gamma_{iN}(V_N(x_N)),
\gammaeo_{i1}(||e_1||), \ldots,\gammaeo_{iN}(||e_N||)
)\;.
\end{multline}
The inequality above and (\ref{c1}) yield that for each $i\in
{\cal N}(x)$
\begin{multline}
V_i(x_i)> \mu_i(\gamma_{i1}(\sigma_1(r)),..,\gamma_{iN}(\sigma_N(r)), \varphi_{i1}(r),.., \varphi_{iN}(r))\\
 \ge  \mu_i(\gamma_{i1}(V_1(x_1)),\ldots,\gamma_{iN}(V_N(x_N)),
\gammaeo_{i1}(||e_1||), \ldots,\gammaeo_{iN}(||e_N||))\;.
\end{multline}
Hence, by (\ref{aa1}),
\[
\nabla V_i(x_i) f_i(x, g_i(x+e))\le - \alpha_i(||x_i||)
\]
for all $i\in {\cal N}(x)$.\\
We now provide a bound to $\langle p, f_i(x, g_i(x+e))\rangle $
for each $p\in \partial \sigma_i^{-1}(V_i(x_i))$ and $i\in {\cal N}(x)$. Observe that
$\sigma^{-1}$ is only locally Lipschitz and the Clarke generalized
gradient must be used for $\sigma_i^{-1}(V_i(x_i))$.  Fix $x_i$
and let $\rho>0$ be such that $||x_i||=\rho$. Define the compact set
$K_\rho=\{V_i(x_i)\in \R_+:\rho/2\le ||x_i||\le 2\rho\}$, and
let
\[
c_\rho=\min_{r\in K_\rho} (\sigma_i^{-1})'(r)\;,\quad C_\rho=\max_{r\in
K_\rho} (\sigma_i^{-1})'(r)\;,
\]
where $c_\rho>0$  by definition of the $\Omega$-path $\sigma$. Bearing
in mind that $||x_i||=\rho$, for each $p\in
\partial \sigma_i^{-1}(V_i(x_i))$  there exists $\gamma_\rho\in
[c_\rho, C_\rho]$ such that $p=\gamma_\rho\nabla V_i(x_i)$, and
$\langle p, f_i(x, g_i(x+e))\rangle =\gamma_\rho \nabla V_i(x_i)\cdot f_i(x,
g_i(x+e))\le - \gamma_\rho
\alpha_i(\rho)\le -c_\rho \alpha_i(\rho)$.\\
Set $\tilde \alpha_i(\rho):= c_\rho
\alpha_i(\rho)$, which is a positive function for all positive
$\rho$. Also set
\[
\alpha(r):=\min\{\tilde \alpha_i(||x_i||)\,:\,
r=||x||\;,\; i\in {\cal N}(x) \}\;.
\]
Then, for each $p\in \partial \sigma_i^{-1}(V_i(x_i))$, $\langle p,
f_i(x, g_i(x+e))\rangle \le -\tilde \alpha_i(||x_i||)\le -\alpha(||x||)$.
This in turn implies (\cite{drw}) that for each $p\in \partial V(x)$ $\langle p, f(x,
g(x+e))\rangle \le -\alpha(||x||)$.
\end{proof}
\medskip
In the rest of the section we discuss an event-triggered control
scheme %(\cite{tabuada.tac07})
for the system \eqref{def:trigsys} with  triggering conditions
that ensure that the condition on the state $x$ and the error $e$ as in Theorem~\ref{t1} are satisfied.

\begin{theorem}
\label{t3a}
Let Assumptions \ref{a1} and \ref{a2} hold. Consider
the interconnected system
\be\label{is} \dot x_i(t)=f_i(x(t), g_i(\hat x(t)))\;,\quad i\in
{\cal N}\;,
\ee
as in \eqref{def:trigsys} with triggering conditions given by
\[T_i(x_i,e_i)=\chi_i(\|e_i\|)-V_i(x_i)\,,\]
with $\chi_i$ defined in \eqref{eq:threestar} for all $i\in\mathcal{N}$.
Assume that no Zeno behavior is induced, i.e. the sequence of times $t_k$,
where the $t_k$'s are defined by the triggering conditions $T_i$ as discussed in Section~\ref{sec:trig}, has no accumulation point or is a finite
sequence for all $i\in\mathcal{N}$.
Then the
origin is a globally uniformly asymptotically stable equilibrium
for (\ref{is}).
\end{theorem}
\begin{proof}
To analyze the event-based control scheme introduced above, we
define the time-varying map $\tilde f(t,x)=f(x,g(x+e(t)))$.
The map $\tilde
f(t,x)$ satisfies the Carath\'eodory conditions
for the
existence of solutions (see e.g.\ \cite{bacciotti.rosier.book},
Section 1.1). Because of the conditions on $f$  (see
Section \ref{s.prelim}), the solution exists and is unique. Along
the solutions of $\dot x=\tilde f(t,x)$, the locally Lipschitz
positive definite and radially unbounded Lyapunov function $V(x)$
introduced in Theorem \ref{t1} satisfies
\[
V(x(t''))-V(x(t'))=\int_{t'}^{t''} \dst\frac{d}{dt} V(x(t)) dt
\]
for each pair of times $t''\ge t'$ belonging to the interval of
existence of the solution. Moreover, by a property of the Clarke
generalized gradient (\cite{tesi.ceragioli}, Section 2.3,
Proposition 4), for almost all $t\in \R_+$, there exists $p
\in \partial V(x(t))$ such that:
\[
\dst\frac{d}{dt} V(x(t))= \langle p,\tilde f(t,x(t)) \rangle\;.
\]
Note that the triggering conditions $T_i(x_i,e_i)=\chi_i(\|e_i\|)-V_i(x_i)\geq0$
ensures that $V_i(x_i)\geq\chi_i(\|e_i\|)$ for all positive times. Hence we can use Theorem~\ref{t1} together with
the definition of $\tilde f(t,x)$,
to infer (see \cite{sanfelice.et.al.tac07}, Section IV.B,
for similar arguments)
\[
V(x(t''))-V(x(t'))\leq-\int_{t'}^{t''} \alpha(||x(t)||) dt\;.
\]
We can now apply \cite{bacciotti.rosier.book}, Theorem 3.2, to
conclude that the origin of $\dot x=\tilde f(t,x)$, and therefore
of $\dot x=f(x,g(\hat x))$,  is uniformly globally
asymptotically stable.
\end{proof}
 The
assumption that no Zeno behavior is induced is quite strong. One
possibility is to cast the event-triggering approach in the
framework of hybrid systems and study the asymptotic stability of
the system in the presence of Zeno behavior (see
\cite{RG-RS-AT:09}, pp. 72--73 for a discussion in that respect).
Another possibility is to extend the solution (\cite{ames.et.al}).
Let $t^\ast$ be the accumulation time such that
$\lim_{k\to+\infty} t_k=t^\ast$. Since the Lyapunov function is
decreasing along the solution over the interval of time $[0,
t^\ast)$, then $\lim_{k\to+\infty} V(x(t_k))$ exists and is
finite. Let us denote this limit value as $V^\ast$. If $V^\ast\ne
0$, then one can pick a state $x^\ast$ such that
$V(x^\ast)=V^\ast$ and consider the solution to the system
(\ref{is}) with initial condition $x^\ast$. If Zeno behavior
appears again,  one can repeat indefinitely the same argument  and
conclude  that $V(x(t))$ converges to zero either in finite or in
infinite time, with $x(t)$ obtained by the repeated extension of
the solution after the Zeno times. However, this approach may
raise a few implementation issues, such as the detection of the
Zeno time and the choice of the new initial condition $x^\ast$ at
the Zeno time, and may discourage to follow this path. For this
reason, slightly different triggering conditions which rule out
the possibility of Zeno behavior are introduced in
Section~\ref{sec.towards}.

\section{An example}\label{sect.linear}

Consider the interconnection of linear systems as in Section~\ref{sec:ex}
\begin{equation}
    \label{eq:sysex}
 \dot x_i=\sum_{j=1}^N\bar A_{ij}x_j+\sum_{j=1}^N \bar B_{ij}e_j\quad i\in\mathcal{N}\;,
\end{equation}
with $\bar A_{ii}$ Hurwitz for $i\in\mathcal{N}$. In order to
apply our event-triggered sampling scheme, we first have to check
the conditions of Theorem~\ref{t1}. As
verified in Section~\ref{sec:ex}, Assumption~\ref{a1} holds for
system \eqref{eq:sysex} with each Lyapunov function given by
$V_i(x_i)=x_i^\top P_ix_i$.\\
To check Assumption~\ref{a2} we recall Lemma~7.2 from \cite{drw}:
\begin{lemma}
 \label{lem:7.2}
 Let $\alpha\in\K_\infty$ satisfy $\alpha(ab)=\alpha(a)\alpha(b)$ for all $a,b\geq 0$. Let $D=\operatorname{diag}(\alpha)$,
 $G\in\R^{n\times n}$, and $\Gamma_\mu$ be given by
 \[\Gamma_\mu(s)=D^{-1}(GD(s))\,.\]
 Then $\Gamma_\mu\ngeq\operatorname{id}$ if and only if the spectral radius of $G$ is less than one.
\end{lemma}
It is easy to see that for the linear case $\Gamma_\mu$ from Section~\ref{sec:lyap} with entries from \eqref{eq:gain_lin} is of the form of Lemma~\ref{lem:7.2} with $\alpha(r):=\sqrt{r}$ and
\[G_{ij}=\frac{2||P_i||^{3/2}}{\tilde c_i}\frac{||\bar A_{ij}||}{[\lambda_{\min}(P_j)]^{1/2}},\quad i\neq j,\;i,j\in\mathcal{N}\]
and zeros as diagonal entries. In other words, $\gamma_{ij}(r)=G_{ij}\alpha(r)$.
Let us assume that the spectral radius of $G$ is less than one.
For the case of linear systems an $\Omega$-path is given by a half line in the direction of an eigenvector $s^*$ of a matrix
$G^*$ which is a perturbed version of $G$ (for details, see the proof of \cite[Lemma~7.12]{drw}). Denote this half line by $\sigma(r):=s^*r$.\\
To show a way to construct a $\varphi$ for which
\begin{equation}
\label{cond:ex1}
\overline{\Gamma}_\mu(\sigma(r), \varphi(r))<\sigma(r)\;,\quad
\forall r>0
\end{equation}
holds, consider the $i$th row of \eqref{cond:ex1} and exploit the fact that the $\Omega$-path is linear:
\[ \left(\sum_{j=1,j\ne i}^N\frac{2||P_i||^{3/2}}{\tilde c_i}\frac{||\bar A_{ij}||}{[\lambda_{\min}(P_j)]^{1/2}}
\sqrt{rs^*_j}+\sum_{j=1}^N\varphi_{ij}(r)\right)^2\hspace*{-7pt}<rs^*_i.
\]
Bearing in mind that $||\bar A_{ij}||\ne 0$ if and only if $j\in
\Sigma(i)$ or $j\in C(i)$ (see last paragraph of Example~\ref{sec:ex}
together with the definition of the set $\Sigma$), the inequality can be
rewritten as
\begin{multline}
    \label{rwr} \left(\sum_{j\in (\Sigma(i)\cup
    C(i))\setminus{i}}\frac{2||P_i||^{3/2}}{\tilde c_i}\frac{||\bar
    A_{ij}||}{[\lambda_{\min}(P_j)]^{1/2}}
  \sqrt{rs^*_j}+\right.
\left.\sum_{j=1}^N\varphi_{ij}(r)\right)^2\hspace{-4pt}<rs^*_i.
\end{multline}

If we make the choice $\varphi_{ij}(r)=a_{ij} \sqrt r$ for all $j\in C(i)$
and $\varphi_{ij}(r)=0$ otherwise, we obtain
\begin{equation}
\sum_{j\in C(i)} a_{ij}
<\sqrt{s^*_i}- \sum_{j\in(\Sigma(i)\cup C(i))\setminus{i}}\frac{2||P_i||^{3/2}}{\tilde c_i}\frac{||\bar A_{ij}||}{[\lambda_{\min}(P_j)]^{1/2}}\sqrt{s^*_j}
=:\rho_i.
\end{equation}

It is worth noting that $ \rho_i>0$ by the spectral condition on $G$.\\
Assume without loss of generality that $(\Sigma(i)\cup C(i))\setminus{i}\neq \emptyset$ (if not,
(\ref{rwr}) trivially holds). Note that it would be sufficient to assume irreducibility of the interconnection structure to ensure $(\Sigma(i)\cup C(i))\setminus{i}\neq \emptyset$.\\
Without further knowledge of the system, we choose for $j\in C(i)$ for the gains
$\varphi_{ij}(r):=\frac{\rho_i}{|C(i)|}\sqrt{r}$, where
$|C(i)|$ denotes the cardinality of the set $C(i)$,
to ensure that \eqref{cond:ex1} holds. If $j\notin C(i)$ set $\varphi_{ij}=0$.
Simulations suggest that it might be better to not choose the $a_{ij}$ uniformly, but to relate them to the system matrices (in particular,
to the spectral radii of the coupling matrices $\bar B_{ij}=B_iK_{ij}$).\\
Now we can calculate the trigger functions $\chi_i$ as in Theorem~\ref{t1} by using the $\Omega$-path and the $\varphi_{ij}$ from above.
The map $\hat\gammaeo_i$ is calculated using the $\gammaeo_{ij}$ from \eqref{eq:gain_lin}.
Stability of the interconnected system is then inferred by Theorem~\ref{t3a}.\\
\begin{figure}[htbp]
 \centering
 \begin{minipage}[t]{6 cm}
     \includegraphics[width=6 cm]{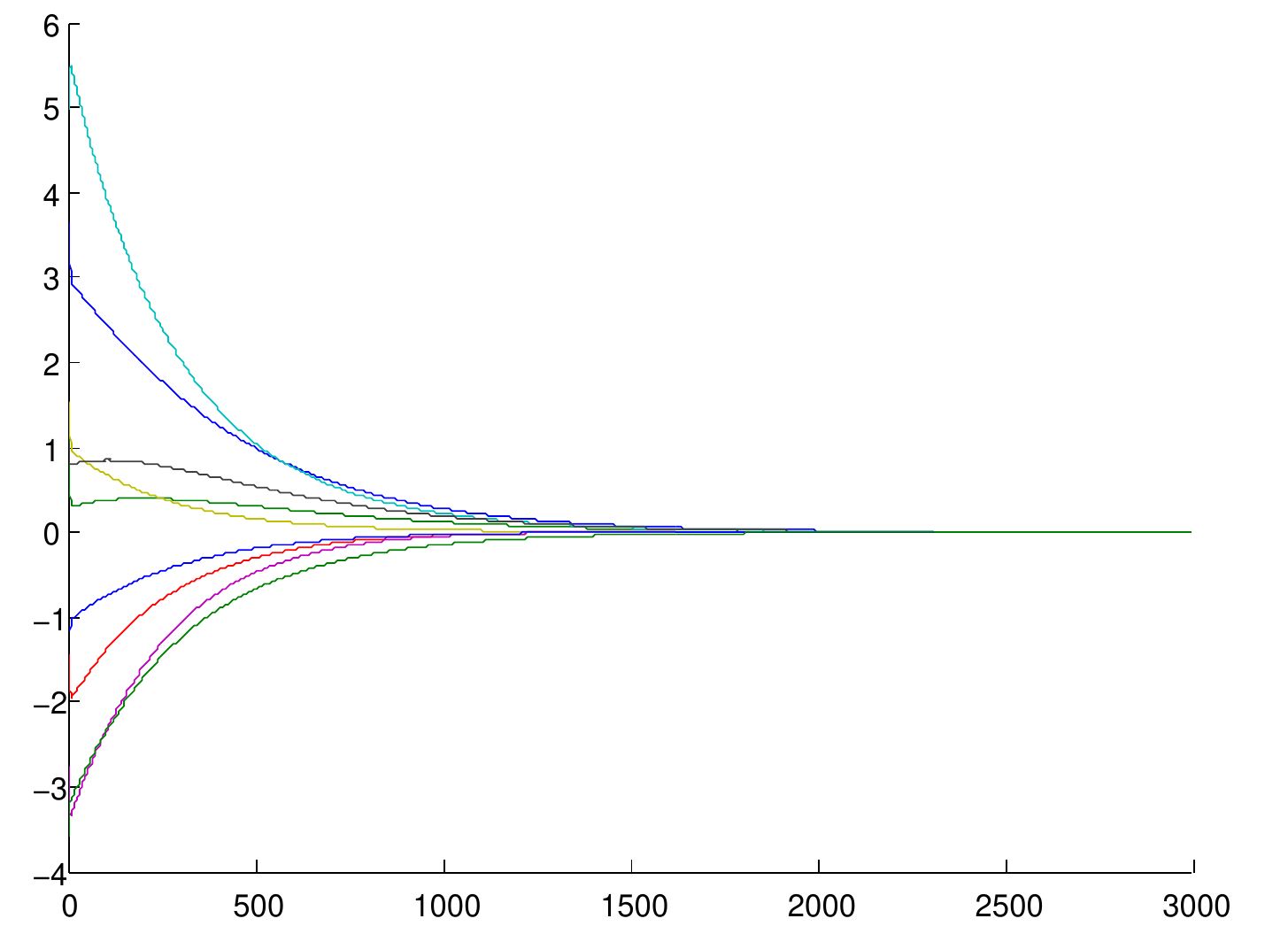}
      \put(-170,60){\tiny{\begin{sideways} $x$ \end{sideways}}}
              \put(-20,10){\tiny{$t$}}
   \caption{Trajectories of the interconnected system with periodic sampling}
   \label{fig:periodic}
 \end{minipage}
 \begin{minipage}[t]{6 cm}
    \includegraphics[width=6 cm]{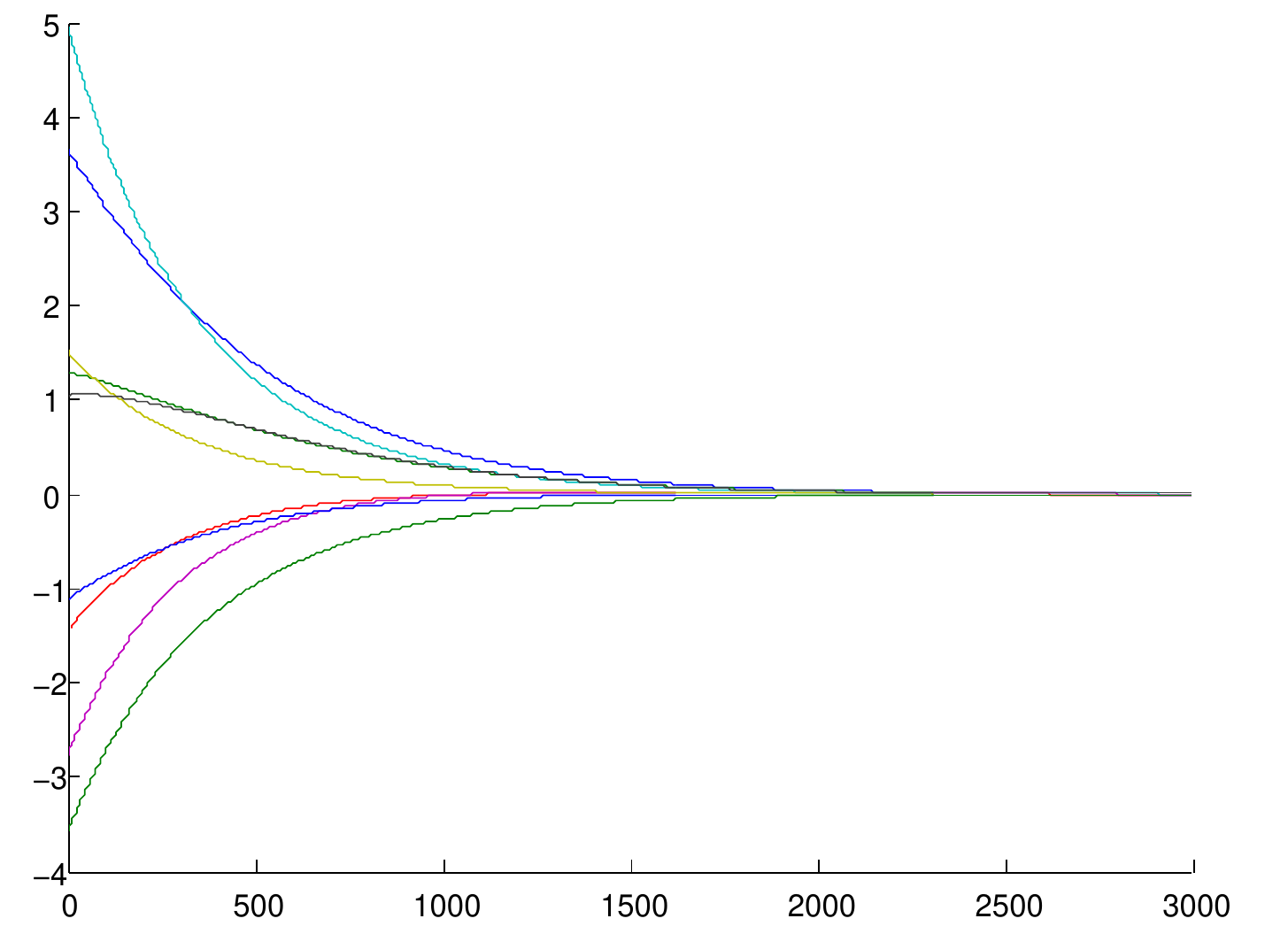}
    \put(-170,60){\tiny{\begin{sideways} $x$ \end{sideways}}}
                  \put(-20,10){\tiny{$t$}}
   \caption{Trajectories of the interconnected system with event-triggering}
   \label{fig:event}
 \end{minipage}
\end{figure}
\begin{figure*}[htb]
\centering
%\begin{minipage}[b]{6cm}
\includegraphics[width=0.8\textwidth]{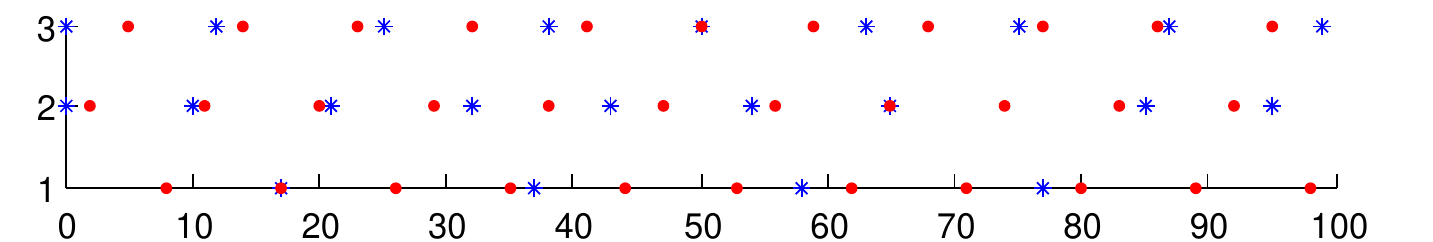}
\put(-275,17){\tiny{\begin{sideways} system \end{sideways}}}
              \put(-33,0){{$t$}}
\caption{\label{fig:begin}33 periodic (red dots) and 22 (blue stars) events at the beginning of the simulation}
%\end{minipage}
%\begin{minipage}[b]{6cm}
\centering
\includegraphics[width=0.8\textwidth]{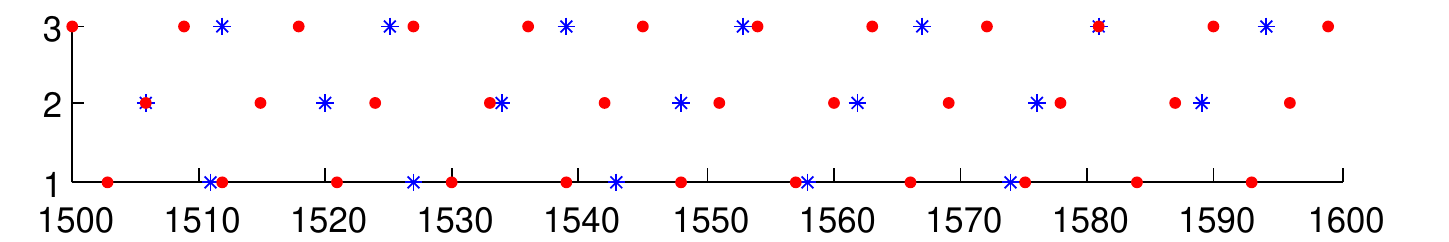}
\put(-275,17){\tiny{\begin{sideways} system \end{sideways}}}
              \put(-33,0){{$t$}}
\caption{\label{fig:middle}34 (red dots) periodic and  19 (blue stars) events in the middle of the simulation}
%\end{minipage}
\end{figure*}
\begin{figure*}[htbp]
\centering
\includegraphics[width=0.8\textwidth]{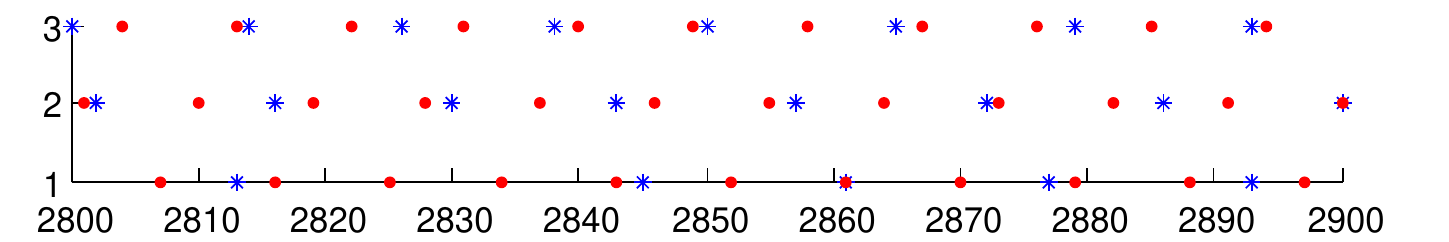}
\put(-275,17){\tiny{\begin{sideways} system \end{sideways}}}
              \put(-33,0){{$t$}}
\caption{\label{fig:end}34 periodic (red dots) and 21 (blue stars) events at the end of the simulation}
\end{figure*}
To illustrate the feasibility of our approach we simulated the
interconnection of three linear systems of dimension three.  The entries
of the system matrices are drawn randomly from a uniform distribution on
the open interval $(-5,5)$.
We repeat this procedure until the spectral radius of the corresponding matrix $G$ is less than one.\\
In Figure~\ref{fig:periodic} new information is sampled every three units of time. Which system has to transmit information is decided by a round robin protocol (i.e., first system one, than system two, system three and again system one and so on). \\
In Figure~\ref{fig:event} our event-triggered sampling scheme is used.\\
Over the range of $3000$ units of time the system with periodic sampling
transmitted $1000$ new information, whereas in our scheme the events were
triggered only $595$ times. By looking at Figure~\ref{fig:periodic} and
Figure~\ref{fig:event} it seems like the periodic sampled system converges
a bit faster. %reaches a given threshold of $\|x\|<0.023$ a bit earlier.
Indeed, the systems state norm of the event-triggered system at time
$3000$ is already reached by the periodic sampled system after $2486$
(i.e., $828$ periodic samplings) units of time.
But still the number of triggered events ($595$) is smaller than the number of periodic events ($828$).\\
A representation of how the the different systems (1,2, or 3) sample their
state is depicted in Figures~\ref{fig:begin}-\ref{fig:end}. The first
picture shows the sampling behavior at the beginning ($t\in[0,100]$) of
the simulation. The other two are from the middle ($t\in[1500,1600]$) and
the end ($t\in[2800,2900]$) of the simulation, respectively.  There is a
small overshoot for some of the trajectories in
Figure~\ref{fig:periodic}. This behavior cannot be seen in
Figure~\ref{fig:event}, because in the event-triggered implementation
information is transmitted more frequently at the beginning by systems 2
and 3, whereas in the periodic implementation transmission starts (for
systems 2 and 3) a few samples later, as can be seen in
Figure~\ref{fig:begin}. The reason for this is that we set $\hat x=0$
instead of initializing $e_0=0$. This is another possibility of
initializing the controller (and hence the initial error) than the one
described in Remark~\ref{rem:trigg}.

\section{A nonlinear example}\label{sect.nonlinear.example}
\label{sec:nonlinear}
The following interconnection of $N=2$ subsystems
\[
\ba{rcl}
\dot x_1 &=& x_1x_2 +x_1^2 u_1\\
\dot x_2 &=& x_1^2 + u_2
\ea
\]
is considered  under the assumption that each controller  can only access the state of the system it controls.
The control laws are chosen accordingly as
\[
u_1=-(x_1+e_1)\;,\; u_2=-k(x_2+e_2)\;,\; k>0\;.
\]
Let $V_i(x_i)=\frac{1}{2} x_i^2$ for $i=1,2$. Then
\begin{equation}
\dot V_1(x_1):=\nabla V_1(x_1)(-x_1^3 + x_1 x_2-x_1^2 e_1)
\leq x_1^2(-\frac{1}{2}x_1^2+|x_2|+\frac{1}{2}e_1^2)
\end{equation}
from which we can deduce
\[
\frac{1}{4}x_1^2\ge |x_2|+\frac{1}{2}e_1^2\quad \Rightarrow \quad \dot V_1(x_1)\le - \frac{1}{4}x_1^4\,.
\]
Since the left-hand side of the implication is in turn implied by
$V_1(x_1)\ge \max\{\sqrt{32V_2(x_2)}, 2e_1^2\}$, this shows that the
first subsystem fulfills Assumption~\ref{a1} with
\begin{multline}
\mu_1=\max\;,\; \gamma_{11}(r)=0\;,\; \gamma_{12}(r)=\sqrt{32r}\;,\; \eta_{11}(r)=2r^2\;,\; \eta_{12}(r)=0\;.
\end{multline}
Similarly
\[
\dot V_2(x_2):=\nabla V_2(x_2)(x_1^2 -k x_2-k e_2)\leq |x_2|(-k|x_2|+x_1^2+k |e_2|)
\]
and therefore
\[
V_2(x_2)\ge \max\{\frac{32}{k^2} V_1^2(x_1), 8 e_2^2\}\quad \Rightarrow \quad \dot V_2(x_2)\le - \frac{k}{2}x_2^2\;,
\]
i.e. the second subsystem satisfies Assumption~\ref{a1} with
\begin{multline}
\mu_2=\max\;,\; \gamma_{21}(r)=\frac{32}{k^2} r^2\;,\; \gamma_{22}(r)=0\;,\;
 \eta_{21}(r)=0\;,\; \eta_{22}(r)=8 r^2\;.
\end{multline}
As discussed in Section \ref{sec:lyap}, in the case of $N=2$ the $\Omega$-path
can be chosen as $\sigma_1={\rm Id}$ and $\gamma_{21}<\sigma_{2}<\gamma_{12}^{-1}$.
Provided that $k>32$, one can set
$\sigma_2(r)=\overline \sigma^2 r^2$, with $\overline \sigma^2\in (\frac{32}{k^2}, \frac{1}{32})$.
If $\varphi \in ({\cal K}_\infty\cup \{0\})^{N\times N}$ is additionally chosen as
\[
\varphi_{11}(r)=\sqrt{32}\overline \sigma r\;,\; \varphi_{12}\equiv \varphi_{21}\equiv 0\;,\;
\varphi_{22}(r)=\frac{32}{k^2} r^2
\;,
\]
then Assumption \ref{a2} is satisfied. In view of the choice of
$\sigma$, $\mu$ and $\varphi$, the requirement (\ref{omega.condition}) boils down to the condition
$\Gamma_\mu(\sigma(r))<\sigma(r)$ which is equivalent to the small-gain condition $\gamma_{12}\circ \gamma_{21}<{\rm Id}$
(see Section \ref{sec:lyap}). This small-gain condition is fulfilled by the choice of $k$,
since $\gamma_{12}\circ \gamma_{21}(r)=\frac{32}{k} r$. Hence Theorem \ref{t1} applies and
provides an expression for the functions $\chi_i$ used in the event-triggered implementation of the control laws. The functions are given explicitly by
\[\chi_1(r):=\frac{1}{\sqrt{8}\overline \sigma}r^2,\quad
\chi_2(r):=\frac{\overline\sigma^2k^2}{4}r^2
\;.\]
Simulation results for the initial condition $x_{1}(0)=-4,\,x_2(0)=3,\,\hat x_1(0)=-4$ and $\hat x_2(0)=3$
can be found for $t\in[0,0.5]$ in Figures~\ref{fig:periodicnon}--\ref{fig:eventnon}.\\
The trajectory of the first system is given in blue and for the second system in green.
%\margin{what does teal mean?\\ I can count 8 red samples and perhaps 6 green samples. Where does 12 come from?}
The input is calculated using the red and turquoise values accordingly.\\
Figure~\ref{fig:eventnon} shows the event triggering scheme from Theorem~\ref{t3a}. After $0.5$ seconds $12$ events were triggered. Recall that we start with $e(0)=0$. This initialization is not counted as an event. The shortest time between two events is $0.0155$ seconds. In Figure~\ref{fig:periodicnon} we used a periodic sampling scheme with a sampling period of $0.0155$ seconds leading to a total of $66$ samples. Compared to the event triggering scheme no major improvement in the performance of the closed loop system can be recognized. Although in the periodically sampled system more than five times the amount of information was transmitted.\\
If we use a Round Robin protocol with $12$ periodic samples as in Figure~\ref{fig:nonlinear1} the sampling scheme is not able to stabilize the system.
\begin{figure}[htb]
 \centering
 \begin{minipage}[t]{6 cm}
     \includegraphics[width=6 cm,height=3.8cm]{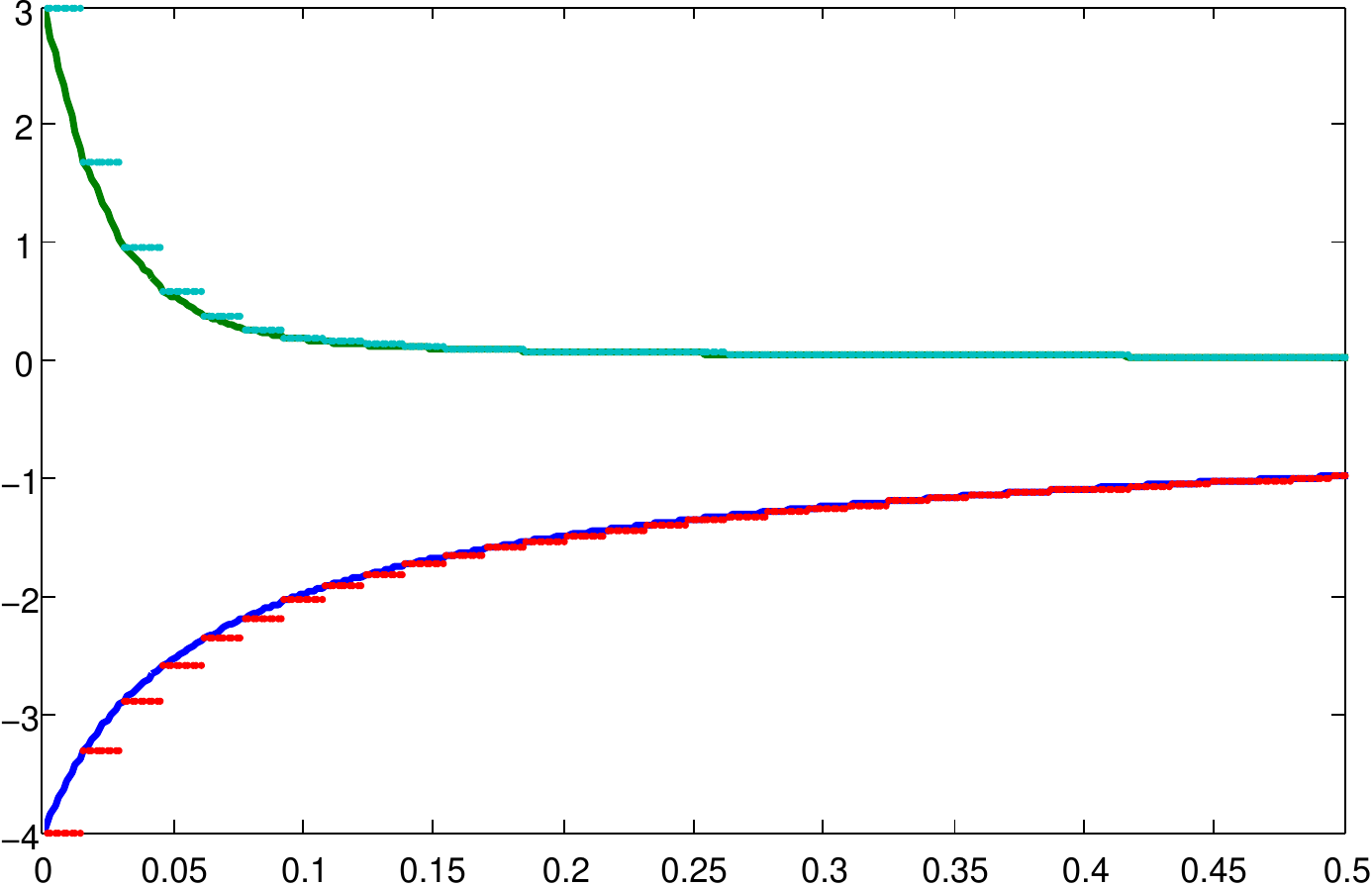}
     \put(-175,55){\tiny{\begin{sideways} $x$ \end{sideways}}}
                   \put(-13,-2){\tiny{$t$}}
   \caption{\label{fig:periodicnon}Trajectories of the interconnected system with periodic sampling (66 samples)}
 \end{minipage}
\begin{minipage}[t]{6 cm}
\includegraphics[width=6 cm,height=3.8cm]{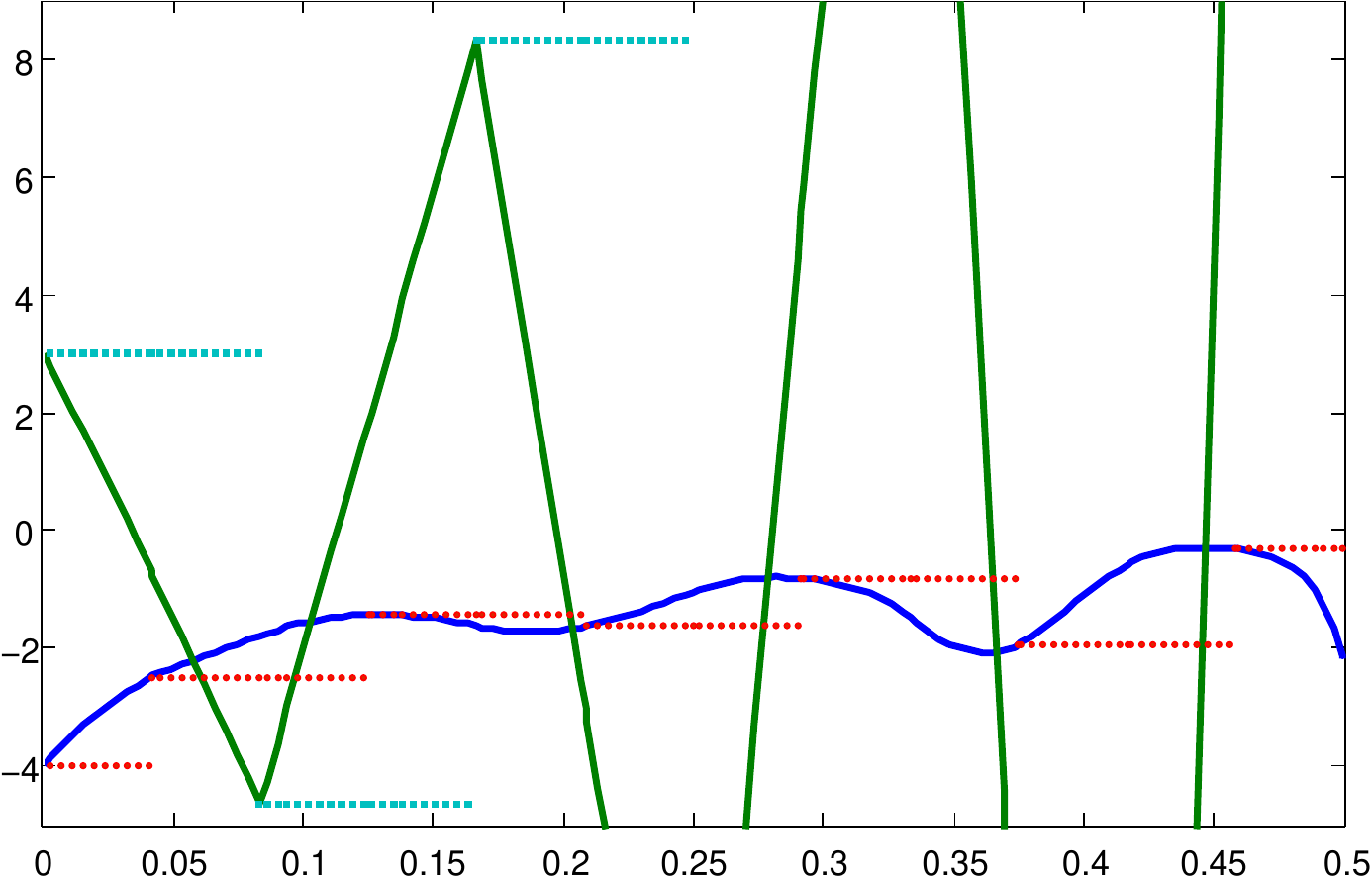}
 \put(-173,55){\tiny{\begin{sideways} $x$ \end{sideways}}}
                   \put(-13,-2){\tiny{$t$}}
\caption{\label{fig:nonlinear1}Trajectories of the interconnected system with periodic sampling (12 samples Round Robin)}
\end{minipage}
\end{figure}
\begin{figure}[htb]
 \centering

\begin{minipage}[t]{6 cm}
    \includegraphics[width=6 cm,height=3.8cm]{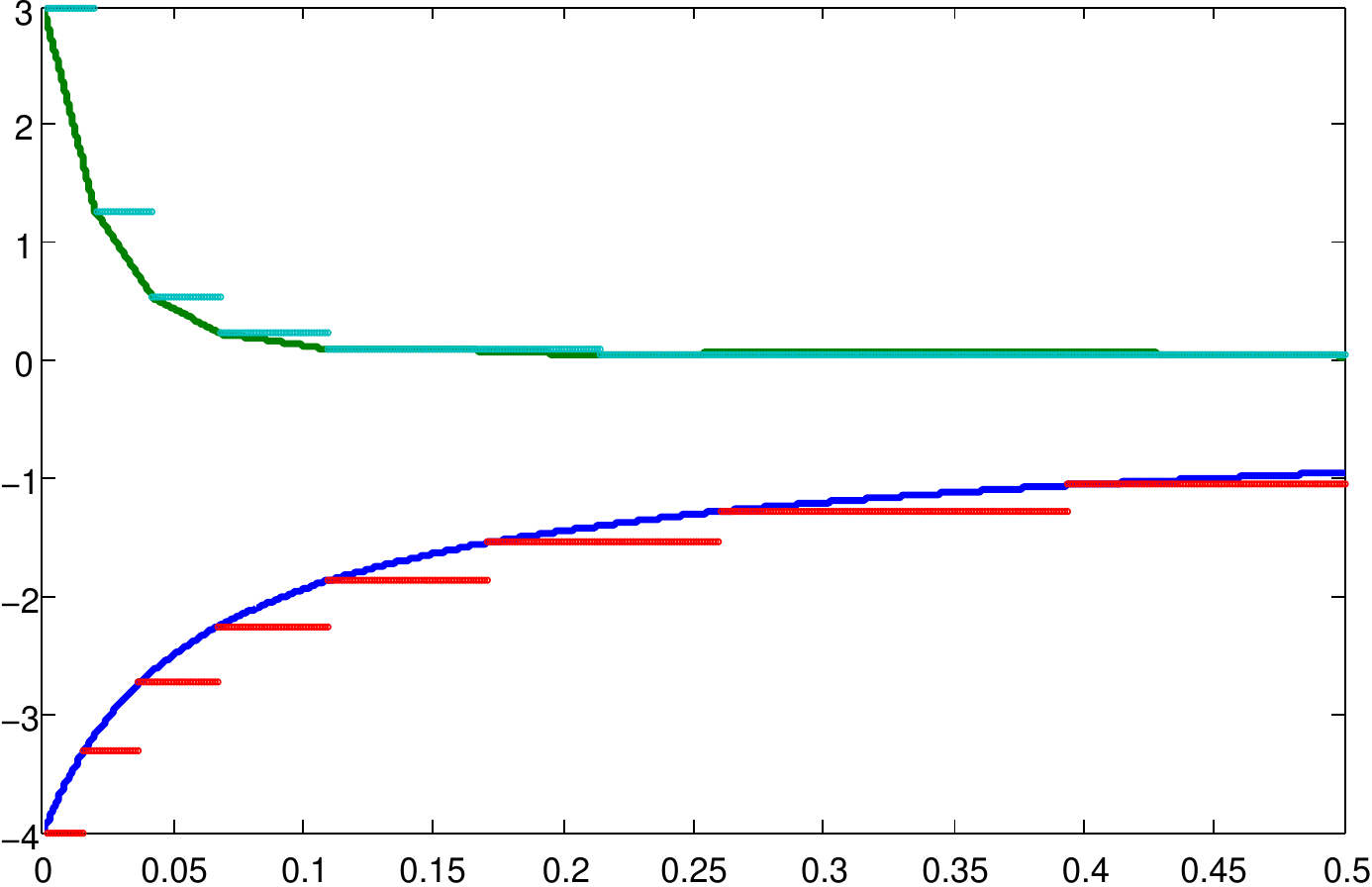}
     \put(-175,55){\tiny{\begin{sideways} $x$ \end{sideways}}}
                       \put(-13,-2){\tiny{$t$}}
   \caption{\label{fig:eventnon}Trajectories of the interconnected system with event-triggered sampling (12 events)}
 \end{minipage}
\begin{minipage}[t]{6 cm}
 \includegraphics[width=6 cm,height=3.8cm]{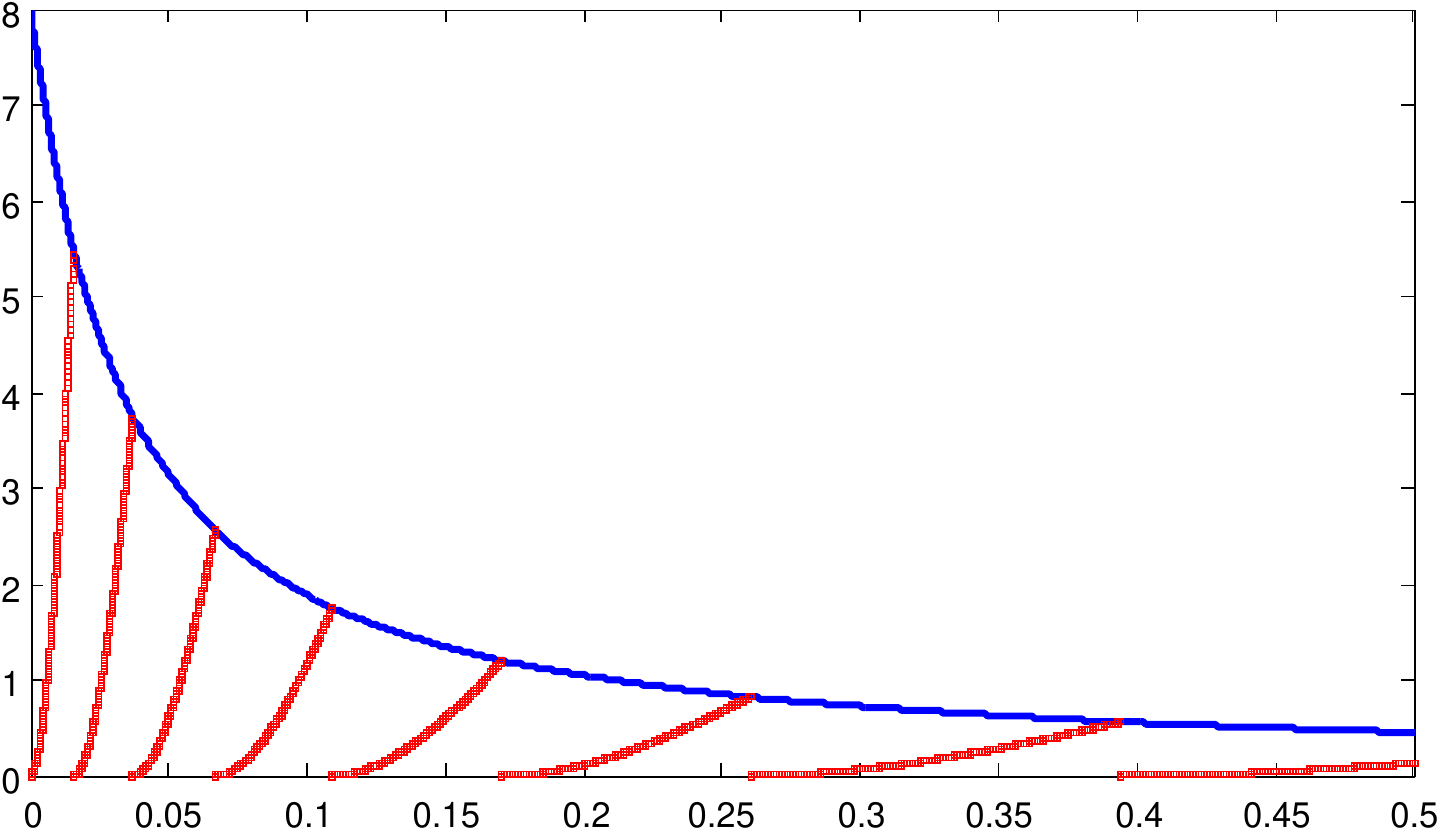}
 \put(-13,-2){\tiny{$t$}}
\caption{\label{fig:evsv}Lyapunov function and $\chi_2(\|e_2\|)$ for the second subsystem}
\end{minipage}
\end{figure}
In Figure~\ref{fig:evsv} the Lyapunov function for the second subsystem
together with $\chi_2(\|e_2\|)$ is plotted. Every time the red curve (the
error function) hits the blue line, the error is reset to zero.
%\newpage
\section{On Zeno-free distributed event-triggered
control}\label{sec.towards} The aim of this section is to show
that it is possible to design distributed event-triggered control
schemes for which the accumulation of the sampling times in finite
time does not occur. The focus is again on  the system
(\ref{interconnected.system}), namely:
\be\label{interconnected.system.2} \ba{rcl} \dot x_i &=& f_i(x,
g_i(x+e))\;. \ea\ee
We present two different approaches for a Zeno free event-triggered control scheme.
The first is based on a practical ISS-Lyapunov assumption whereas the second tries to lower the amount of event-triggering by reducing unneeded information transmissions.
\subsection{Practical Stabilization}
\label{sec:practical}
Here we adopt a slight variation of the input-to-state
stability property from Assumption~\ref{a1}.
\begin{assumption}\label{a4}
For $i\in\mathcal{N}$, there exist a differentiable function
$V_i: \R^{n_i}\to \R_{+}$, and class-${\cal K}_\infty$ functions
$\alpha_{i1}, \alpha_{i2}$ such that
\[
\alpha_{i1}(||x_i||)\le V_i(x_i)\le \alpha_{i2}(||x_i||)\;.
\]
Moreover there exist functions $\mu_i\in {\rm MAF}_{2N}$,
$\gamma_{ij},\eta_{ij}\in {\cal K}_\infty$, for $j\in\mathcal{N}$,
positive definite functions $\alpha_i$ and positive constants $c_{i}$, for
$i\in\mathcal{N}$, such that
%
%\begin{equation}

\begin{multline}\label{aa5p}
 V_i(x_i)\ge \max\{\mu_i(\gamma_{i1}(V_1(x_1)),
\ldots, \gamma_{iN}(V_N(x_N)), \eta_{i1}(||e_1||), \ldots,
\eta_{iN}(||e_N||)),c_i\}\\
\Rightarrow  \nabla V_i(x_i) f_i(x, g_i(x+e))\le -
\alpha_i(||x_i||)\;.
\end{multline}
%\end{equation}
\end{assumption}
\begin{remark}
    Systems satisfying Assumption~\ref{a4} are usually referred to as
    input-to-state practically stable (ISpS) (\cite{jiang.et.al.aut96}).
\end{remark}
We now state a new version of Theorem \ref{t1} for system
(\ref{interconnected.system.2}).
\begin{theorem}\label{t3}
Let Assumptions \ref{a2} and \ref{a4} hold. Let
$V(x)=\max_{i\in{\cal N}} \sigma_i^{-1}(V_i(x_i))$.  Assume that for each $j\in{\cal N}$ ,
\be\label{str2} \max\{\sigma_j^{-1}(V_j(x_j)),c_j\}\geq\hat\eta_j(\|e_j\|)\;,
\ee
where
\begin{equation}
 \label{eq:etahat}
\hat \eta_{j}=\max_{i\in Z(j)}\varphi^{-1}_{ij}\circ\eta_{ij}\;.
\end{equation}
Then there exists a positive definite $\alpha:\R_+\rightarrow\R_+$ such that
\[
\langle p,f(x,g(x+e))\rangle\le -\alpha(||x||),%+\gamma(|w|)
\; \forall p \in \partial V(x)\;,
\]
for all $x=(x_1^\top\,x_2^\top\,\ldots\,x_N^\top)^\top\;\in\;\{x:
V(x)\ge\hat c:=\max_i\{c_i,\sigma_i^{-1}(c_i)\} \}$,
where
\[
f(x,g(x+e)) = \left(\ba{c}
f_1(x, g_1(x+e))\\
\ldots\\
f_N(x, g_N(x+e))\\
\ea\right)\;.
\]
\end{theorem}
\begin{proof}
Let ${\cal
N}(x)\subseteq {\cal N}$
 be the indices $i$ such that $V(x)=
\sigma_i^{-1}(V_i(x_i))$.\\
Take any pair of indices $i,j\in {\cal N}$. By definition,
$V(x)\ge \sigma_j^{-1}(V_j(x_j))$ and
\be\label{circle}
\gamma_{ij}(\sigma_j(V(x)))\ge \gamma_{ij}(V_j(x_j))\;. \ee Let
$i\in {\cal N}(x)$. Then by Assumption \ref{a2},
we have:
\begin{multline}\label{wh}
V_i(x_i)=\sigma_i(V(x))>\\\mu_i(\gamma_{i1}(\sigma_{1}(V(x))),\ldots,
\gamma_{iN}(\sigma_{N}(V(x))), \varphi_{i1}(V(x)),\ldots,
\varphi_{iN}(V(x)))\;.
\end{multline}
Bearing in mind (\ref{circle}), we also have
\begin{multline}\label{wh5}
V_i(x_i)=\sigma_i(V(x))>\\\mu_i(\gamma_{i1}(V_1(x_1)),\ldots,
\gamma_{iN}(V_N(x_N)), \varphi_{i1}(V(x)),\ldots,
\varphi_{iN}(V(x)))\;.
\end{multline}
Let us partition the set $\mathcal{N}:=\mathcal{P} \cup \mathcal{Q}$. The
set $\mathcal{P}$ consists of all the indices $i$ for which the first part
of the maximum in condition \eqref{str2} holds,
i.e. $i\in\mathcal{P}:\Leftrightarrow\sigma_i^{-1}(V_i(x_i))\geq c_i$; also
 $\mathcal{Q}:=\mathcal{N}\setminus\mathcal{P}$.  For all
$j\in\mathcal{P}$ we have by \eqref{str2}
$\sigma_j^{-1}(V_j(x_j))\geq\hat\eta_j(\|e_j\|)$ and hence using
\eqref{eq:etahat} (as mentioned before, the case $j\notin C(i)$ is
trivial)
\begin{multline}
 \label{wh6}
 \varphi_{ij}(V(x))\geq\varphi_{ij}\circ\sigma_j^{-1}(V_j(x_j))\geq\varphi_{ij}\circ\hat\eta_j(\|e_j\|)\geq\\\varphi_{ij}\circ\varphi^{-1}_{ij}\circ\eta_{ij}(\|e_j\|)=\eta_{ij}(\|e_j\|)\,.
\end{multline}
Assume now that $V(x)\geq \hat c$. For all $j\in\mathcal{Q}$ we have by \eqref{str2} $c_j\geq\hat\eta_j(\|e_j\|)$ and so
\begin{equation}
 \label{wh7}
 \varphi_{ij}(V(x))\geq\varphi_{ij}(\hat c)\geq\varphi_{ij}(c_j)\geq\varphi_{ij}\circ\hat\eta_j(\|e_j\|)\geq\eta_{ij}(\|e_j\|).
 \end{equation}
 Combining \eqref{wh6} and \eqref{wh7} we get for all $j\in\mathcal{N}$
 \[\varphi_{ij}(V(x))\geq\eta_{ij}(\|e_j\|)\;,\]
provided that (\ref{str2}) holds and that $V(x)\ge \hat c$.
 Substituting the latter in \eqref{wh5} yields
\begin{multline}\label{wh2}
V_i(x_i)=\sigma_i(V(x))>\\\mu_i(\gamma_{i1}(V_1(x_1)),\ldots,
\gamma_{iN}(V_N(x_N)),  \eta_{i1}(||e_1||),\ldots,
\eta_{iN}(||e_N||))\;.
\end{multline}
Because $i\in\mathcal{N}(x)$ and $V(x)\geq\hat c=\max_i\{c_i,\sigma_i^{-1}(c_i)\}$,
we have\\ $V(x)=\sigma_i^{-1}(V_i(x_i))\geq \hat c\geq \sigma_i^{-1}(c_i)$ and finally we conclude
\[V_i(x_i)\geq c_i\;.\]
The latter together with \eqref{wh2}
is the left-hand side of the implication (\ref{aa5p}).
Hence, $\nabla V_i(x_i) f_i(x, g_i(x+e))\le - \alpha_i(||x_i||)$
for all $i\in {\cal N}(x)$.
Now we can repeat the same arguments of the last part of the proof
of Theorem \ref{t1}, and conclude that for all $x$ such that
$V(x)\ge \hat c$ and for
all $p\in \partial V(x)$, $\langle p,f(x,g(x+e))\rangle\le
-\alpha(||x||)$.
\end{proof}
Now that we established an analog to Theorem~\ref{t1} for the case
of practical stability, we are able to infer stability of the
closed loop event-triggered control system \eqref{def:trigsys}
without the assumption that Zeno behavior does not occur.
%of Zeno freeness.
%\margin{is Zeno freeness
%a commonly used expression?}
%\margin{Rudi: I don't know. Maybe it is too ``german''. Anyway, I changed it }
\begin{theorem}
 \label{t5}
Let Assumptions \ref{a2} and \ref{a4} hold. Consider
the interconnected system
\be\label{is2a} \dot x_i(t)=f_i(x(t), g_i(\hat x(t)))\;,\quad i\in
{\cal N}\;, \ee
as in \eqref{def:trigsys} with triggering conditions given by
\begin{equation}
 T_i(x_i,e_i)=\hat\eta_i(\|e_i\|)-\max\{\sigma^{-1}_i\circ V_i(x_i),\hat c_i\}\,,
\label{eq:triggcond1}
\end{equation}
with $\hat\eta_i$ defined in \eqref{eq:etahat} for all
$i\in\mathcal{N}$. Then the origin is a globally uniformly
practically stable equilibrium for (\ref{is2a}).
\end{theorem}
\begin{proof}
 Here we want to adopt the same line of reasoning as in the proof of Theorem~\ref{t3a}. To this end, we have to make sure that by the triggering conditions given by \eqref{eq:triggcond1} no Zeno behavior is induced. Note that in between triggering events $\dot e(t) = -\dot x(t)$ for all $t\in(t_k,t_{k+1})$ by the definition of \eqref{def:trigsys}.  The triggering conditions
$T_i(x_i,e_i)=\hat\eta_i(\|e_i\|)-\max\{\sigma^{-1}_i\circ
V_i(x_i),\hat c_i\}\geq 0$ ensure that $\max\{\sigma^{-1}_i\circ
V_i(x_i),\hat c_i\}\geq\hat\eta_i(\|e_i\|)$ for all positive
times.
Following the same reasoning of Theorem~\ref{t3a}, with the
exception that we have to replace the application of
Theorem~\ref{t1} with the application of Theorem~\ref{t3}, one
proves that $V(x(t))$ is decreasing along the solution $x(t)$ on
its domain of definition. Hence,  $x(t)$ is bounded  on
its domain of definition. Since
$\max\{\sigma^{-1}_i\circ V_i(x_i(t)),\hat
c_i(t)\}\geq\hat\eta_i(\|e_i(t)\|)$, then also
$e(t)$ is bounded and so is $\hat x(t)=x(t)+e(t)$. As
$e_j(t_k^+)=0$ for each index $j$ which triggered an event and
$\dot e(t)$ is bounded in between events ($\dot e(t)=-\dot
x(t)=-f(x(t),g(\hat x(t)))$), the time when the next event will be
triggered by system $j$ is bounded away from zero because the time
it takes $e_j$ to evolve from zero to $c_j$ is bounded
away from zero. Hence, either there is a finite number of  times $t_k$
or $t_k\rightarrow\infty$ as $k$ goes to infinity. The solution $x(t)$ is then defined
for all positive times and Theorems~\ref{t3a} and \ref{t3} allow us to conclude.
\end{proof}
%\medskip
In hybrid systems, the practice of avoiding Zeno effects while retaining stability
in the practical sense is referred to as temporal regularization (see \cite{RG-RS-AT:09}, p.\ 73,
and references therein). Here, the regularization is achieved via a notion of practical
ISS. In the context of event-triggered ${\cal L}_2$-disturbance attenuation control
for linear systems temporal regularization is studied in \cite{donkers.heemels.ifac11}.

\subsection{Parsimonious Triggering}

In Section~\ref{sec:practical} a way to exclude the occurrence of Zeno behavior for the price of practical stability rather than asymptotic stability was shown.
Here we want to provide a way to exclude the Zeno effect by introducing a new triggering scheme, but still achieving asymptotic stability. The main idea behind the new triggering scheme, which will be introduced in Theorem~\ref{t7b} is that if the error of the $i$th subsystem is bigger than its Lyapunov function but still small compared to the Lyapunov function of the overall system, no transmission of the $i$th subsystem is needed. \\
For future use we need also a slight variation of Theorem~\ref{t1}.
Here we exploit the fact that we can either compare each state to its corresponding error (as in Theorem~\ref{t1}) or each error to the Lyapunov function of the overall system as shown in the next theorem.
\begin{theorem}
 \label{t2b}
Let Assumptions \ref{a1} and \ref{a2} hold. Let $V(x)=\max_{i\in
{\cal N}} \sigma_i^{-1}(V_i(x_i))$ and, for each $j\in {\cal N}$,
define:
\begin{equation}
\label{eq:twostar}
\hat\gammaeo_j=\max_{i\in {Z(j)}}
\varphi_{ij}^{-1}\circ\gammaeo_{ij}\;.
\end{equation}
Then there exist a positive definite $\alpha\,:\,\R_+\rightarrow\R_+$ such that the condition
\be\label{cond:iuj}
V(x)\geq\hat\eta_j(\|e_j\|),\;\forall j\in\mathcal{N}
\ee
implies
\[
\langle p,f(x,g(x+e))\rangle\le -\alpha(||x||),%+\gamma(|w|)
\; \forall p \in \partial V(x)\;.
\]
\end{theorem}
\begin{proof}
 The proof follows by a slight modification of the proof of Theorem~\ref{t1}.\\
For each $x$, let $\mathcal{N}(x)\subset\mathcal{N}$ be set of indices for which $V(x)=\sigma_i^{-1}(V_i(x_i))$.\\
It is sufficient to show that for all $i\in\mathcal{N}(x), \,j\in\mathcal{N}$ we have
$\varphi_{ij}(r)\geq\eta_{ij}(\|e_j\|)$, with $r=V(x)$.\\
First recall that for $j\notin C(i)$ the latter inequality trivially holds. So assume that $j\in C(i)$.
 Using \eqref{eq:twostar} and \eqref{cond:iuj}, we have
%Here we have by using the definition of $J(x)$ and \eqref{eq:twostar} that
\begin{equation*}
\varphi_{ij}(V(x))\geq\varphi_{ij}(\hat\eta_{j}(\|e_j\|))\geq\varphi_{ij}\circ\varphi_{ij}^{-1}\circ\eta_{ij}(\|e_j\|)=\eta_{ij}(\|e_j\|)\;.
\end{equation*}
i.e.\ \eqref{varphi} in the proof of Theorem~\ref{t1}. Then the argument after \eqref{varphi}
can be repeated word by word. By previous arguments this concludes the proof.
\end{proof}
A triggering condition for the $j$th subsystem which yield the validity of condition \eqref{cond:iuj} would make the knowledge of the Lyapunov function $V$ of the overall system to system $j$ necessary.
This would contradict our wish for a decentralized approach.\\
The next lemma provides a decentralized way to ensure that condition \eqref{cond:iuj} holds. To this end, we give an approximation of the other states (the $W$) which will be compared to the error instead of the Lyapunov function of the overall system. Appropriately scaled, $W$ is a lower bound on the Lyapunov function of the overall system and hence can be used to check the validity of \eqref{cond:iuj}. The important advantage is, that this approximation can be calculated by using only local information.
Before we state the next lemma, define
\[\xi^{j,x_j}:=\left(\xi_1^\top,\dots,\xi_{j-1}^\top,x_j^\top,\xi_{j+1}^\top,\dots,\xi_N^\top\right)^\top\;\]
as the vector $\xi$ where the $j$th component is replaced by $x_j$. The proofs of Lemma~\ref{lem:theta}, \ref{lem:zeno}, and \ref{lem:xdotx2} are postponed to the Appendix.
\begin{lemma}
\label{lem:theta}
Let Assumptions~\ref{a1} and \ref{a2} hold. Let $V(x)=\max_{i\in
{\cal N}} \sigma_i^{-1}(V_i(x_i))$.
 Let $d_j\in\R_+$ be an approximation of $\|f_j(x,g_j(x+e))\|$ with
$|\|f_j(x,g_j(x+e))\|-d_j|\leq\tilde\kappa_j\|x_j\|$.\\
Assume that for $j\in\mathcal{N}$ there exist functions    $\Theta_j\;:\R^N\rightarrow\R$ such that $V(x)\geq\hat\eta_i(\|e_i\|)$ for all $i\neq j$ implies
\begin{equation}
\label{cond:lem22}
\Theta_j(\|x_1\|,\dots,\|x_N\|)\geq\|f_j(x,g_j(x+e))\|\,.
\end{equation}
Define
\[W(j,x_j,d_j)=\min\{\max_{i\neq j}\|\xi_i\|\;:\;\xi\in \mathcal{A}(j,x_j,d_j)\}\]
with
\begin{equation}
 \label{eq:setA2}
    \begin{aligned}
        \mathcal{A}(j,x_j,d_j)=
        \{\xi^{j,x_j}\;:\;\Theta_j&(\|\xi_1\|,\dots,\|\xi_N\|)\geq
       & d_j-\tilde\kappa_j\|x_j\|\}\;.
    \end{aligned}
\end{equation}
%\margin{Define $\wedge$ is the notation section.\\
%Rudi: as it appears only once, I just changed it to "and"}
Then the conditions
\begin{equation}
\label{cond:lem2}
W(j,x_j,d_j)\geq\psi_j^{-1}\circ\hat\eta_j(\|e_j\|)\quad\text{and}
\quad \sigma_j^{-1}(V_j(x_j))\leq\hat\eta_j(\|e_j\|),
\end{equation}
with $\psi_j=\max_{i\neq j}\sigma_i^{-1}\circ\alpha_{i1}$
imply \[V(x)\geq\hat\eta_j(\|e_j\|)\;.\]
%FRW \alpha_{1i} ??? please check. I changed it to \alpha_{i1} in line
% with the notation from Assumption 1.
\end{lemma}
\begin{remark}
    Lemma~\ref{lem:theta} presents a way of approximating the norm of the
    other states influencing the dynamics of a single subsystem. To this
    end an approximation of the derivative is used. Another possibility for
    achieving this goal would be to construct an observer, which gives an
    approximation of the inputs (the other states) by observing the
    dynamics of a single subsystem.
\end{remark}
Before we can state another event-triggering scheme, which does not induce
Zeno behavior we have to formulate the observation that if Zeno behavior
occurs, one of the states has to approach the equilibrium.
\begin{lemma}
\label{lem:zeno}
Consider a large scale system with triggered control of the form
\eqref{def:trigsys} satisfying Assumptions~\ref{a1} and \ref{a2}. Let
$\chi_i,\,i\in\mathcal{N}$ be given by \eqref{eq:threestar}.  Consider the
triggering conditions given by
\[
T_i(x_i,e_i)=\chi_i(\|e_i\|)-V_i(x_i)\;.
\]
If for a given initial condition $x_0$ the triggering scheme $T_{i^*}$ induces Zeno behavior, then the corresponding solution
%there exists an index $i\in\mathcal{N}$  such that
\[x_{i^*}(t_k)\rightarrow0\;.\]
\end{lemma}
The next lemma provides an inequality for the state and the corresponding
dynamics.  Besides the rather technical nature of Lemma~\ref{lem:xdotx2},
together with Lemma~\ref{lem:zeno} it forms the basis to be able to
compare the $i$th state to the rest of the states as will be seen in
Theorem~\ref{t7b}.
%FRW: reformulated lemma xdotx2
\begin{lemma}
\label{lem:xdotx2}
Consider system
\begin{equation}
\dot x=f(x,g(x+e))
\label{eq:syslem3}
\end{equation}
as in \eqref{def:trigsys}. If there are triggering instances $t_k \to t^*$
for $k\to \infty$ and an index $i$ such that $x_{i}(t_k) \to 0$, then for
all $M>0$ there exists a $k^*\in\N$ such that for some $k\geq k^*$
\[\frac{\|x_i(t_{k+1})-x_i(t_k)\|}{t_{k+1}-t_k}> M\|x_i(t_{k+1})\|.\]
%FRW: I commented this out, because this follows without ado from
%t_k \to t^*, or not
%Furthermore, for all $k\geq k^*$ it holds that
%\begin{equation}
%\label{eq:mtk}
%t_k-t_{k-1}<\frac{1}{M}\;.
%\end{equation}
\end{lemma}

It is of interest to note the following immediate corollary.

\begin{corollary}
    \label{c:noGAS}
    Under the conditions of Lemmas \ref{lem:zeno} and
    \ref{lem:xdotx2}, assume that the functions $\Theta_j$ from
    Lemma~\ref{lem:theta} satisfying \eqref{cond:lem22} may be chosen to
    be Lipschitz and so that $\Theta_i(0,\dots,0)=0$ holds. Consider an
    initial condition $x(0)=x_0 \neq 0$. If there is Zeno behavior at
    $t^*$, i.e. if there are triggering instances $t_k \to t^*$, then for
    the overall state $x$ of \eqref{eq:syslem3}
    \begin{equation}
        \label{eq:noGAS}
        x(t_k) \not\to 0\quad \text{ as } k\to \infty\,.
    \end{equation}
\end{corollary}

\begin{proof}
    We first exclude that there is a $s^*\in [0,t^*)$ such that $x(s^*) =
    0$.  Otherwise choose Lipschitz constants $L_i$ for $\Theta_i$ valid
    on the compact set $\{ x(s) \;;\; s\in [0,s^*]\}$ and note that we
    have for each $i$ almost everywhere on $[0,s^*]$
    \begin{multline}
    \label{eq:noGAS1}
        \|\dot{x}_i(t)\| = \|f_i(x(t),g_i(x(t)+e(t)))\| \\
    \leq \Theta_i(\|x_1(t)\|,\ldots,\|x_N(t)\|) \leq L_i \|x(t)\|\;.
    \end{multline}
    Note that we can use $\Theta_i$ as a bound for the dynamics as in \eqref{cond:lem22}, because the validity of $V_i(x_i)\geq \chi_i(\|e_i\|)$ for all $i$ trivially implies $V(x)\geq \hat{\eta}_i(\|e_i\|)$.\\
    As \eqref{eq:noGAS1} is true for all $i$, this implies
    $\|\dot{x}(t)\| \leq L \|x(t)\|$ for $L$ sufficiently large and almost
    all $t \in [0,s^*]$. It follows that $\|x(s^*)\| \geq e^{-Ls^*}
    \|x(0)\|>0$, so that $x(s^*)\neq 0$.\\
    If $x(t_k)\to 0$, then $x(t) \to 0$ for $t\nearrow t^*$. Hence for
    each $i$, and $k$ sufficiently large we have that \eqref{eq:noGAS1}
    holds almost everywhere on $(t_k,t^*)$. As in the first part of the
    proof it follows that $\|x(t^*)\| \geq e^{-L(t^*-t_k)} \|x(t_k)\|>0$,
    because by the first step of the proof $x(t_k)\neq 0$. This
    contradicts the assumption that $x(t_k)\to 0$.
\end{proof}

The rest of this section is devoted to constructing an event-triggered
control scheme which ensures that Condition~\eqref{cond:iuj} holds.\\
{}From Lemma~\ref{lem:zeno} we know that if Zeno behavior occurs, then one
of the subsystems approaches the origin in finite time. Corollary~\ref{c:noGAS}
shows that under certain regularity assumptions, a number of subsystems do
not converge to $0$ as we approach the Zeno point. Hence, from a
certain time on, the Lyapunov function corresponding to the
subsystem which tends to
 the origin does not contribute to the Lyapunov function for the overall
system. As a consequence no information transfer from this subsystem
is necessary using parsimonious triggering. This
observation is made rigorous in the rest of the section.\\
%{\color{green} [c] "in the following" means "in the rest of the section"?
%Please make this precise. Moreover, the Remark after this sentence breaks
%the flow of the reading. The Remark should be  relocated or its contents
%included in the main text of the paper. I couldn't figure out a solution.}
%\begin{remark}
%    As the aim is to use only local information, we will use the
%    difference quotients as one way of approximating the size of the
%    derivative at the triggering points. Furthermore, we do not wish to
%    assume that all subsystems are aware of all triggering events. Hence
%    in the following we will use the notation $t_{k}^j$ to denote those
%    triggering events initiated by system $j$. We define
%    \begin{equation}
%        \label{eq:djdef}
%        d_j(t)=
%        \frac{\|x_j(t)-x_j(t_{{k-1}}^j)\|}{t-t_{{k-1}}^j}
%    \end{equation}
%    as the difference quotient approximating $\|\dot{x}_j(t)\|$ after the triggering event $t_{k-1}^j$.
%
%%    Note that if Zeno behavior occurs
%%
%%    \begin{multline}
%%        \|\dot{x}_j(\zeta_k)\| - \|\dot{x}(t_k)\| \leq
%%        \|\dot{x}_j(\zeta_k) - \dot{x}(t_k)\| \\
%%       = \| f_j(x(\zeta_k),g((x+e)(\zeta_k))-f_j(x(t_k),g((x+e)(t_k))\|\\
%%       \leq L| t_k - \zeta_k| \leq L |t_k - t_{k-1}|\,.
%%    \end{multline}
%%    Hence if
%\end{remark}
In the next theorem we use the triggering condition as in Theorem~\ref{t3a} but we add another triggering condition $T_{i2}$, which checks whether the $i$th subsystem contributes to the Lyapunov function of the overall system. It does so by comparing the local error of system $i$ with the approximation $W$ of the other states as described in Lemma~\ref{lem:theta}. The main idea is that if the dynamics of the $i$th system is large compared to its own state, other states must be large. As the correct value of the dynamics is not known to system $i$, an approximation of $\|\dot{x}_i\|$ is used.\\
 As the aim is to use only local information, we will use the
    difference quotients to approximate the size of the
    derivative at the triggering points. Furthermore, we do not wish to
    assume that all subsystems are aware of all triggering events. Hence
    in the following we will use the notation $t_{k}^i$ to denote those
    triggering events initiated by system $i$. We define
    \begin{equation}
        \label{eq:djdef}
        d_i(t)=
        \frac{\|x_i(t)-x_i(t_{{k-1}}^i)\|}{t-t_{{k-1}}^i}
    \end{equation}
    as the difference quotient approximating $\|\dot{x}_i(t)\|$ after the triggering event $t_{k-1}^i$.\\
Adding the new triggering condition that uses \eqref{eq:djdef} allows us to exclude the occurrence of Zeno behavior. A discussion about the new triggering condition can be found in Remark~\ref{rem:parsim} and \ref{rem:approx}.
\begin{theorem}
 \label{t7b}
 Consider a large scale system with triggered control of the form
\eqref{def:trigsys} satisfying Assumptions~\ref{a1} and \ref{a2}. Let
 $V(x)=\max_{i\in\mathcal{N}}\sigma_i^{-1}(V_i(x_i))$.  Define
 \[T_{i1}(x_i,e_i)=\chi_i(\|e_i\|)-V_i(x_i)\] with $\chi_i$ as in
 Theorem~\ref{t3a} and
 \[T_{i2}(x_i,e_i,d_i)=\psi_i^{-1}\circ\hat\eta_i(\|e_i\|)-W(i,x_i,d_i)\;,\] where $\psi_i,\,W(i,x_i,d_i)$ and $\hat\eta_i$ are defined as in Lemma~\ref{lem:theta}. Furthermore, assume that for all $i\in\mathcal{N}$ the $\Theta_i$ from Lemma~\ref{lem:theta} and $\psi_i^{-1}\circ\hat\eta_i$ are Lipschitz with Lipschitz constant $L_i$ respectively $K_i$ and that $\Theta_i(0,\dots,0)=0$ holds.\\
 Consider the interconnected system
\be\label{is2} \dot x_i(t)=f_i(x(t), g_i(\hat x(t)))\;,\quad i\in
{\cal N}\;, \ee
as in \eqref{def:trigsys} with triggering conditions given by
\begin{equation}
 T_i(x_i,e_i,d_i)=\min\{T_{i1}(x_i,e_i),T_{i2}(x_i,e_i,d_i)\}\,,
\label{eq:triggcond2}
\end{equation}
for all $i\in\mathcal{N}$. Then the origin is a globally uniformly asymptotically stable
equilibrium for (\ref{is2}), if there are constants $\kappa_j>0,\;j\in {\cal
  N}$ such that at the triggering times $t_k$, which are implicitly
defined by \eqref{is2} and \eqref{eq:triggcond2} as described in
Section~\ref{sec:trig}, the following condition is satisfied:
\begin{equation}
\label{eq:condapprox}
|\|\dot x_j(t_k^j)\|-d_j(t_k^j)|\leq\kappa_j\|x_j(t_k^j)\|
\end{equation}
where $d_j(t_{k_j}^j)$ is defined by \eqref{eq:djdef}. In particular, no Zeno behavior occurs.
\end{theorem}
\begin{proof}
 Before we can use Theorem~\ref{t3a} respectively Theorem~\ref{t2b} to conclude stability, we have to exclude the occurrence of Zeno behavior.
First note that condition \eqref{eq:triggcond2} triggers an event if and only if $T_{i1}\geq 0$ and $T_{i2}\geq 0$ respectively condition \eqref{str} and \eqref{cond:iuj} are violated.
Now assume that the $j$th subsystem induces Zeno behavior.
For simplicity, we omit the index $j$ of the triggering times $t_k^j$.
 Hence, let $t_k$ the triggering times of the $j$th subsystem and $t^*=\lim_{k\rightarrow\infty}t_k$ the finite accumulation point.
{}From Lemma~\ref{lem:zeno} we know  that the $j$th subsystem has to approach the equilibrium, i.e.
$\lim_{t_k\rightarrow t^*}x_j(t_k)=0$.
%And hence $V_j(x_j(t^*))=0$ and $e_j(t^*)=0$.\\
Lemma~\ref{lem:xdotx2} tells us that for all $M$ there exists a $k^*$ such that
for some $ k\geq k^*$
\begin{equation}
\label{eq:dotbound}
\frac{\|x_j(t_{k})-x_j(t_{k-1})\|}{t_{k}-t_{k-1}}> M\|x_j(t_k)\|\;.
\end{equation}
% To show that triggering condition $T_{i2}$ inhibits Zeno behavior, take a $\xi(t_k)\in\mathcal{A}(j,x_j(t_k),d_j(t_k))$ as in the definition of $W$.
 As discussed in the proof of Lemma~\ref{lem:theta}, the full state $x\in A\subset\mathcal{A}$. But the knowledge of $x$ is not available to a single subsystem. Hence, we take $\xi(t_k)\in\mathcal{A}(j,x_j(t_k),d_j(t_k))$ as in the definition of $W$ as an approximation for the states of the other subsystems.
 For this $\xi$ we can deduce together with the Lipschitz continuity of $\Theta_j $ and
\eqref{eq:dotbound}
%{}From the Lipschitz continuity of $\Theta_j$, the definition of $W$, and
%\eqref{eq:dotbound} we get
\begin{multline}
\label{eq:lip}
L_j\max_{i\neq j}\{\|\xi_i(t_k)\|,\|x_j(t_k)\|\}\geq
\Theta_j(\|\xi_1\|,\dots,\|x_j\|,\dots,\|\xi_N\|)\geq\\
\frac{\|x_j(t_k)-x_j(t_{k-1})\|}{t_k-t_{k-1}}-\kappa_j\|x_j(t_k)\|\geq(M-\kappa_j)\|x_j(t_k)\|\;.
\end{multline}
%From the premise of Theorem~\ref{t7b} we know that there exists at least one $i\in\Sigma(j)$ or $i\in C(j)$ for which $\lim_{t\rightarrow t^*}\|x_i(t)\|\geq\tau>0$ holds.
%And hence as $x_j$ approaches zero
%\[\max_{i}\|\xi_i\|\geq\frac{K-2\kappa_j}{L}\|x_j\|\;.\]
%choose a $t'$ which is sufficient near $t^*$ such that
%\[\frac{K-2\kappa_j}{L}>1\;.\]
%So we can conclude
And hence for the $k$ given in \eqref{eq:dotbound}
\begin{equation}
\label{eq:xixj}
\max_{i\neq j}\{\|\xi_i(t_k)\|,\|x_j(t_k)\|\}>
\frac{M-\kappa_j}{L_j}\|x_j(t_k)\|\;.
\end{equation}
Now choose \begin{equation}
\label{eq:mchoice}
M>\max\{\kappa_j+L_j,\kappa_j+L_jK_j\}\;,
\end{equation}
 where $K_j$ is the Lipschitz constant of $\psi_j^{-1}\circ\hat{\eta}_j$.
{}From Lemma~\ref{lem:xdotx2} we know that this choice of $M$ yields a $k^*$ such that we can conclude together with \eqref{eq:xixj} $\max_{i\neq j}\{\|\xi_i(t_k)\|,\|x_j(t_k)\|\}=\max_{i\neq j}\|\xi_i(t_k)\|$ for some $k\geq k^*$. For this $k$ we want to show that the corresponding $t_k$ is not a triggering time.\\
To this end we use \eqref{eq:lip} and \eqref{eq:dotbound} to get
\begin{equation}
\label{eq:xidiff}
\max_{i\neq j}\|\xi_i(t_k)\|\geq\frac{1}{L_j}(1-\frac{\kappa_j}{M})\frac{\|x_j(t_k)-x_j(t_{k-1})\|}{t_k-t_{k-1}}\;.
\end{equation}
% And hence for all $t\geq t'$ it holds that $\|x_j\|<W(j,x_j,e_j,d_j)=\min\{\max_{i\neq j}\|\xi_i\|\;:\;\xi\in\mathcal{A}\}$.
% Using the definition of $\mu_j$ and the fact that $T_{j1}=0$ we get
%\[\mu_j(W(j,x_j,e_j,d_j))>\mu_j(\|x_j\|)\geq\sigma_j^{-1}(V_j(x_j))=\hat\eta_j(\|e_j\|)\;\]
%Triggering condition $T_{j2}=\psi_j^{-1}\circ\hat{\eta_j}(\|e_j\|)-W(j,x_j,e_j,d_j)\geq 0$
%We want to show that as $t_k$ goes to $t^*$ the triggering condition $T_{j2}=\psi_j^{-1}\circ\hat{\eta_j}(\|e_j\|)-W(j,x_j,e_j,d_j)$ will be smaller than $0$ for all $k>k^*$.\\
Note that for the $j$th subsystem \eqref{eq:xidiff} is true for all $\xi\in\mathcal{A}$ and therefore by the definition of $W$
\[W(j,x_j,d_j)\geq\frac{1}{L_j}(1-\frac{\kappa_j}{M})\frac{\|x_j(t_k)-x_j(t_{k-1})\|}{t_k-t_{k-1}}\;. \]
Using the latter inequality and the Lipschitz
constant for $\psi_j^{-1}\circ\hat{\eta}_j$ we can bound $T_{j2}$ by
%{}From \eqref{eq:xidiff}, the definition of $W$, and the Lipschitz
%constant for $\psi_j^{-1}\circ\hat{\eta}_j$ we can bound $T_{j2}$ by
\[T_{j2}\leq K_j\|e_j(t_k)\|-\frac{1}{L_j}(1-\frac{\kappa_j}{M})\frac{\|x_j(t_k)-x_j(t_{k-1})\|}{t_k-t_{k-1}}\,.\]
{}From the definition of $e_j(t_k)=x_j(t_{k-1})-x_j(t_k)$ we arrive at
\begin{equation*}
T_{j2}\leq
K_j\|x_j(t_k)-x_j(t_{k-1})\|-\frac{1}{L_j}(1-\frac{\kappa_j}{M})
\frac{\|x_j(t_k)-x_j(t_{k-1})\|}{t_k-t_{k-1}}\,.
\end{equation*}

We may assume that $k^*$ is sufficiently large so that
$t_k-t_{k-1}<M^{-1}$ for all $k\geq k^*$. Together with \eqref{eq:mchoice}
we obtain
\begin{equation*}
    K_j<\frac{1}{L_j(t_k-t_{k-1})}(1-\frac{\kappa_j}{M})
\end{equation*}
and hence $T_{j2}<0$ in contradiction to the assumption that $t_k$ is a
triggering time. Because the only further assumption on the solution of
\eqref{is2} and \eqref{eq:triggcond2} is the occurrence of Zeno behavior,
the aforementioned contradiction shows that Zeno behavior cannot occur.\\
To conclude stability define
\[I(x,e):=\{j\in\mathcal{N}\;:\;V_j(x_j)\geq\chi_j(\|e_j\|)\}\,,\]
\[J(x,e):=\{j\in\mathcal{N}\;:\;V(x)\geq\hat\eta_j(\|e_j\|)\}\,,\]
and
\[\mathcal{J}(x,e):=\{j\in\mathcal{N}\;:\;\psi_j(W(j,x_j,d_j))\geq\hat\eta_j(\|e_j\|)\}\;.\]
Note that the triggering condition $T_j$ ensures that $j\in I \cup \mathcal{J}$.
For $j\in I$ we can use exactly the same reasoning as in Theorem~\ref{t3a}.\\
Lemma~\ref{lem:theta} tells us that from $j\not\in I$ and $j\in\mathcal{J}$ we can deduce $j\in J$. For the case $j\in J$ we can adopt nearly the same reasoning as in Theorem~\ref{t3a}. Only the reasoning for the existence of a Lyapunov function for the overall system changes. In Theorem~\ref{t3a} it can be deduced from Theorem~\ref{t1} whereas here we have to use Theorem~\ref{t2b} to conclude the existence of a Lyapunov function. The rest of the proof  can be copied word by word from Theorem~\ref{t3a}. This ends the proof.
\end{proof}
\begin{remark}
\label{rem:parsim}
 The advantage of parsimonious triggering is twofold.
% \margin{In case reviewers ask some evidence, do you have it?\\Rudi: as the triggering condition $T_i$ combines the old and the new one, it can't have more transmissions. Right now I don't have an example, therefore I just wrote "may"}
 First it allows us to exclude the occurrence of Zeno behavior and second it
 may lead to fewer transmissions compared to the triggering condition given in Theorem~\ref{t3a}.
 Compared to Theorem~\ref{t5} where the same goal is achieved by the notion of practical stability, here we still achieve asymptotic stability, but we have to place more technical assumptions on the involved class $\K$ estimates.\\
Note that the set of indices $j\in\mathcal{N}$ for which condition \eqref{str} holds is a subset of those for which \eqref{cond:iuj} holds (in other words $I(x,e)\subset J(x,e)$). But we cannot check condition \eqref{cond:iuj} locally.\\
 Because of the conservatism we introduce by using $T_{i2}$ instead of \eqref{cond:iuj}, triggering condition $T_i$ still makes sense.\\
In a practical implementation $T_{i1}$ should be checked first, before $T_{i2}$ is checked, because of the conservatism of $T_{i2}$ and the possible cumbersome calculation of $W$.
\end{remark}
\begin{remark}
\label{rem:approx}
One possible drawback of the triggering condition given in Theorem~\ref{t7b} is that the condition on the approximation $d_j$ as in \eqref{eq:condapprox} might be too demanding.
First note, that if $t_k-t_{k-1}$ is sufficiently small, \eqref{eq:condapprox} trivially holds true, because $d_j$ is the difference quotient. As $x_i$ approaches zero, it could happen that the difference $t_k-t_{k-1}$ does not decline fast enough to ensure that \eqref{eq:condapprox} holds.
A way to overcome this issue would be to adjust condition $T_{i2}$ in such a way that it always tries to trigger an event as soon as it cannot be guaranteed that the approximation $d_j$ satisfies \eqref{eq:condapprox}.
\end{remark}
\section{Conclusion}
We presented event-triggered sampling schemes for controlling
interconnected systems. Each system in the interconnection decides
when to send new information across the network independently of
the other systems. This decision is based only on each system's
own state
and a given Lyapunov function. Stability of the interconnected
system is inferred by the application of a nonlinear small-gain
condition. The feasibility of our approach is presented with the
help of numerical simulations. To prevent the accumulation of the
sampling times in finite time, we propose two variations of the
event-triggered sampling-scheme. The first is based on the notion
of input-to-state practical stability, whereas the second compares the local error to an approximation of the Lyapunov function of the overall system to guarantee stability of the interconnected system.
%which guarantee practical
%stability of the interconnected system.
%\nocite{*}

\appendix
\section{Appendix}
\subsection{Proof of Lemma~\ref{lem:theta}}
\begin{proof}
For later use define
\begin{multline}
 \label{eq:setA}
A(j,x_j,e_j,\dot x_j)=\{\xi^{j,x_j}\;:\;\exists\;\epsilon\in\R^{n_i}\text{ s.t. }f_j(\xi^{j,x_j},g_j(\xi^{j,x_j}+\epsilon^{j,e_j}))=\dot x_j\text{ and }\\ V(\xi^{j,x_j})\geq\hat\eta_i(\|\epsilon_i\|)\;\forall i\neq j\}\;.
\end{multline}
The set $A(j,x_j,e_j,\dot x_j)$ describes the set of all  $\xi^{j,x_j}$ for which a pair $(\xi^{j,x_j},\epsilon^{j,e_j})$ exists that fulfills the right hand side of the $j$th subsystem for a given $\dot x_j,x_j,e_j$ and for which $V(\xi^{j,x_j})\geq\hat\eta_j(\|\epsilon_i\|)$ for all $i\neq j$ hold. As the system's state satisfies the dynamics, it holds that $x\in A$.\\
Before we proceed, we want to show that $A(j,x_j,e_j,\dot
x_j)\subset\mathcal{A}(j,x_j,d_j)$.  To this end take a $\xi\in
A(j,x_j,e_j,\dot x_j)$. Hence we have
$f_j(\xi^{j,x_j},g_j(\xi^{j,x_j}+\epsilon^{j,e_j}))=\dot x_j$. Taking the
norm and using \eqref{cond:lem22} yields
\begin{align*}
    \Theta_j(\|\xi_1\|,\dots,\|\xi_n\|)\geq\|f_j(\xi^{j,x_j},g_j(\xi^{j,x_j}+\epsilon^{j,x_j}))\| =\|\dot x_j\|\geq d_j-\tilde\kappa_j\|x_j\|\,,
\end{align*}
where the last inequality follows from the condition on the approximation for $\|\dot x_j\|$. And we can conclude $A(j,x_j,e_j,\dot x_j)\subset\mathcal{A}(j,x_j,d_j)$.\\
{}From condition \eqref{cond:lem2} we can deduce
\begin{multline}
\label{eq:lem21}
\psi_j^{-1}\circ\hat\eta(\|e_j\|)\leq W(j,x_j,d_j)=
\min\{\max_{i\neq j}\|\xi_i\|\;:\;\xi\in \mathcal{A}(j,x_j,d_j)\}\leq\\
\min\{\max_{i\neq j}\|\xi_i\|\;:\;\xi\in A(j,x_j,e_j,\dot x_j)\}\leq\max_{i\neq j}\|x_i\|\;.
\end{multline}
The second inequality follows from $A(j,x_j,e_j,\dot x_j)\subset\mathcal{A}(j,x_j,d_j)$ and the last can be deduced from $x\in A$.
Now we can rewrite \eqref{eq:lem21} to get
\[\hat\eta_j(\|e_j\|)\leq\psi_j(\max_{i\neq j}\|x_i\|)\;.\]
With the help of \eqref{cond:lem2} and the definition of $\psi_j$ we arrive at
\begin{equation*}
    \sigma_j^{-1}(V_j(x_j))\leq\hat\eta_j(\|e_j\|)\leq\max_{i\neq j}\sigma_i^{-1}\circ\alpha_{1i}(\|x_i\|) \leq\max_{i\neq j}\sigma_i^{-1}(V_i(x_i))\,,
\end{equation*}
where the last inequality follows from Assumption~\ref{a1}.  Considering
the first and the last term in the chain of inequalities above it is easy
to see that the $j$th subsystem does not contribute to the Lyapunov
function of the overall system and we conclude $\max_{i\neq
  j}\sigma_i^{-1}(V_i(x_i))=\max_{i\in\mathcal{N}}\sigma_i^{-1}(V_i(x_i))=V(x)$
and the proof is complete.
\end{proof}

\subsection{Proof of Lemma~\ref{lem:zeno}}
\begin{proof}
 Denote $t^*=\lim_{k\rightarrow\infty}t_k$. By definition of the triggering condition we have for each $k$ an index $i(k)\in\mathcal{N}$ such that
 \[V_{i(k)}( x_{i(k)}(t_k))=\chi_{i(k)}(\|e_{i(k)}(t_k)\|)\;.\] Choose
 $i^*\in\mathcal{N}$ such that $i(k)=i^*$ for infinitely many $k$. Such a
 $i^*$ exists because $\mathcal{N}$ is finite and $k$ ranges over all of
 $\N$. Let $K$ be the set of indices for which $i(k)=i^*$. For ease of
 notation let $K=\{s_1,s_2,\dots\}$. By Theorem~\ref{t1} $V$ is a Lyapunov
 function for the event triggered system on the interval $[0,t^*)$.  Thus
 the trajectory $x_{|[0,t^*)}$ is bounded and $e_{|[0,t^*)}$ is bounded
 because $\chi_i(\|e_i(t)\|)\leq V_i(x_i(t))$ for all
 $i\in\mathcal{N},\,t\in[0,t^*)$.
 It follows that $u_{i|[0,t^*))}$ is bounded and so $\dot x_i$ is bounded on $[0,t^*)$ for all $i\in\mathcal{N}$. \\
 Then we have by uniform continuity of $x_{i^*}$ on $[0,t^*)$ that the
 following limit exists
\begin{equation}
 \label{eq:limit}
\lim_{k\rightarrow\infty}\chi_{{i^*}}(\|e_{i^*}(s_k)\|)=\lim_{k\rightarrow\infty}V_{i^*}(x_{i^*}(s_k))=V_{i^*}(x_{i^*}(t^*)).
\end{equation}
By definition $e_{i^*}(s_k^+)=0$. By \eqref{def:trigsys} we have that $\dot e_{i^*}=-\dot x_{i^*}$ almost everywhere on $(s_k,s_{k+1})$. Since $\dot x_{i^*}$ is bounded and $s_{k+1}-s_{k}\rightarrow 0$,
then condition $e_{i^*}(s_k^+)=0$ implies that
\[e_{i^*}(s_{k+1})=e_{i^*}(s_k^+)+\int_{s_k}^{s_{k+1}}\dot e_{i^*}(\tau)d\tau=\int_{s_k}^{s_{k+1}}\dot e_{i^*}(\tau)d\tau\]
which goes to $0$ for $k\rightarrow\infty$.
Hence by \eqref{eq:limit} we obtain that $V_{i^*}(x_{i^*}(t^*))=0$. This shows the assertion.
\end{proof}

\subsection{Proof of Lemma~\ref{lem:xdotx2}}
\begin{proof}
    The proof will be by contradiction.  To this end assume that for some
    fixed $M>0$ and all $k$ sufficiently large we have
\begin{equation}
\label{eq:contra3}
\|x_i(t_{k+1})-x_i(t_k)\|\leq M(t_{k+1}-t_k)\|x_i(t_{k+1})\|\;.
\end{equation}
The evolution of $x_i$ between $t_l$ and $t_k$ for $k>l$ can be bounded by
using a telescoping sum, the triangle inequality, applying
\eqref{eq:contra3}, and a judicious addition of $0$:
\begin{multline*}
\|x_i(t_{k})-x_i(t_l)\|\leq\\\sum\limits_{j=l+1}^{k-1}M(t_{j}-t_{j-1})\|x_i(t_j)-x_i(t_l)\|+M(t_k-t_l)\|x_i(t_l)\|+M(t_k-t_{k-1})\|x_i(t_k)-x_i(t_l)\|\;.
\end{multline*}
If we choose $D>0$ and a $k'$ such that $0\leq\frac{1}{1-M(t_k-t_{k-1})}\leq D$ for all $k>k'$, we can rewrite the latter to
\begin{multline*}
\|x_i(t_{k})-x_i(t_l)\|\leq
\frac{1}{1-M(t_k-t_{k-1})}\sum\limits_{j=l+1}^{k-1}M(t_{j}-t_{j-1})\|x_i(t_j)-x_i(t_l)\|+\\\frac{M(t_k-t_l)}{1-M(t_k-t_{k-1})}\|x_i(t_l)\|\;.
\end{multline*}
Using the discrete Gronwall inequality (see e.g., Theorem 4.1.1 from \cite{agarwal2000}) yields
\begin{multline*}
    \|x_i(t_{k})-x_i(t_l)\|\leq
    \frac{M(t_k-t_l)}{1-M(t_k-t_{k-1})}\|x_i(t_l)\|+\\\frac{1}{1-M(t_k-t_{k-1})}\sum\limits_{j=l+1}^{k-1}\frac{M(t_j-t_l)}{1-M(t_j-t_{j-1})}\|x_i(t_l)\|\;    M(t_{j}-t_{j-1})
    \prod\limits_{s=j+1}^{k-1}(1+\frac{M(t_{s}-t_{s-1})}{1-M(t_s-t_{s-1})})\;.
\end{multline*}
Exploiting that $t_k$ is a monotone sequence, that  $0\leq\frac{1}{1-M(t_k-t_{k-1})}\leq D$ and that $1+x\leq e^x$ for all $x\in\R$
%we obtain
%\begin{multline*}
%\|x_i(t_{k})-x_i(t_l)\|\leq\\
%{\tiny{\left(MD(t_{k}-t_l)+D^2\sum\limits_{j=l+1}^{k-1}
%M^2(t_{k-1}-t_l)(t_{j}-t_{j-1})e^{MD(t_{k-1}-t_{l})}\right)\|x_i(t_l)\|\;.}}
%\end{multline*}
and collapsing the telescoping sum again gives
\begin{multline*}
    \|x_i(t_{k})-x_i(t_l)\|\leq \underbrace{(MD(t_{k}-t_l)+M^2D^2(t_{k-1}-t_l)^2e^{MD(t_{k-1}-t_l)})}_{:=C}\|x_i(t_l)\|\,.
\end{multline*}
Because of the finite accumulation point $t^*$, there exists an $k^*>k'$
such that
\[\|x_i(t_{k})-x_i(t_l)\|\leq C\|x_i(t_l)\|\]
for all $k\geq l\geq k^*$ with $C<1$. Realizing that this contradicts $\|x_i(t^*)\|=0$ finishes the proof.
\end{proof}

\bibliographystyle{plain}
\bibliography{event}
\end{document}